\documentclass[a4paper, 10pt]{amsart}

\usepackage{amsmath,amssymb, hyperref, color}

\newtheorem{theorem}{Theorem}[section]
\newtheorem{lemma}[theorem]{Lemma}

\newtheorem{proposition}[theorem]{Proposition}
\newtheorem{conjecture}[theorem]{Conjecture}
\newtheorem{problem}[theorem]{Problem}

\theoremstyle{definition}

\newtheorem{examples}[theorem]{Examples}
\newtheorem{definition}[theorem]{Definition}

\newcommand{\N}{\mathbb N}
\newcommand{\Z}{\mathbb Z}
\newcommand{\R}{\mathbb R}
\newcommand{\Q}{\mathbb Q}

\newcommand{\red}{{\text{\rm red}}}

\newcommand{\BF}{\text{\rm BF}}

 \DeclareMathOperator{\ord}{ord}

\DeclareMathOperator{\spec}{spec} \DeclareMathOperator{\supp}{supp}
\DeclareMathOperator{\Pic}{Pic} 
 
\DeclareMathOperator{\End}{End} \DeclareMathOperator{\Min}{Min}
\DeclareMathOperator{\Int}{Int}

\newcommand{\DP}{\negthinspace : \negthinspace}

\renewcommand{\t}{\, | \,}

\numberwithin{equation}{section}

\begin{document}

\title{Sets of Lengths}

\address{University of Graz, NAWI Graz \\ Institute for Mathematics and Scientific Computing \\ Heinrichstra{\ss}e 36\\ 8010 Graz, Austria}

\email{alfred.geroldinger@uni-graz.at}

\author{Alfred Geroldinger}

\thanks{
I would like to thank Daniel Smertnig, Salvatore Tringali, Qinghai Zhong, and all the anonymous referees for their careful reading. Their comments helped me to eliminate a couple of flaws, to improve the presentation of the paper, and to increase its readability. Furthermore, I would like to thank the  Scott Chapman for his encouragement to write this survey and all his patience, and Moshe Roitman for showing to me that transfer Krull monoids need not be Mori. \newline This work was supported by the Austrian Science Fund FWF, Project Number P 28864-N35.}



\begin{abstract}
Oftentimes the elements of a ring or semigroup $H$ can be written as finite products of irreducible elements, say $a=u_1 \cdot \ldots \cdot u_k = v_1 \cdot \ldots \cdot v_{\ell}$, where the number of irreducible factors is distinct. The set $\mathsf L (a) \subset \N$ of all possible factorization lengths of $a$ is called the set of lengths of $a$, and the full system $\mathcal L (H) = \{ \mathsf L (a) \mid a \in H \}$ is a well-studied means of describing the non-uniqueness of factorizations of $H$. We provide a friendly introduction, which is largely self-contained, to what is known about systems of sets of lengths for rings of integers of algebraic number fields and  for transfer Krull monoids of finite type as their generalization.
\end{abstract}

\maketitle

\section{Introduction} \label{1}

We all know that every positive integer can be written as a finite product of irreducibles (primes) and that such a factorization is unique up to the order of appearance.
 Similar to  factorizations in the positive integers,
in many rings and semigroups, elements can also be written as finite products  of irreducible elements, but unlike the case of the integers, such factorizations need not always be unique.  It is the main objective of factorization theory to describe the various aspects of non-uniqueness and to classify them in terms of  invariants of the underlying algebraic structure. Before it was extended to commutative ring and semigroup theory, factorization theory had its origin in algebraic number theory,   and only in recent years has been extended to  non-commutative settings \cite{Ba-Sm15, Sm16a}. For further background, we refer the reader to several monographs and conference proceedings  \cite{An97, Ge-HK06a, Ge-Ru09, Fo-Ho-Lu13a, C-F-G-O16}. It is no surprise that this development has been chronicled over  the years by a series of Monthly articles (from \cite{Ja65} to \cite{Ba-Ch11a, Ba-Wi13a}). While the focus  of this interest has been on commutative  domains and their semigroups of ideals, such studies range from abstract semigroup theory to  the factorization theory of motion polynomials  with application in mechanism science \cite{H-L-S-S15}.

Sets of lengths are the most investigated invariants in factorization theory. To fix notation, if an element $a$ in a semigroup  can be written as a product of irreducible elements, say $a=u_1 \cdot \ldots \cdot u_k$, then $k$ is called the length of the factorization, and the set $\mathsf L (a) \subset \N$ of all possible factorization lengths of $a$ is called the set of lengths of $a$.  Under a mild condition on the semigroup, sets of lengths are finite nonempty subsets of $\N_0$, and if there is an element $a$ in the semigroup with $|\mathsf L (a)|>1$ (meaning that $a$ has factorizations of distinct lengths), then sets of lengths can get arbitrarily long (a precise statement of this is in Lemmas \ref{2.1} and \ref{2.2}).

The goal of this paper is to give a friendly introduction to factorization theory. Indeed, we take the reader on a tour through sets of lengths which is highlighted by  two structure theorems (Theorems \ref{2.6} and \ref{5.3}), two open problems (Problem \ref{4.7} and the Characterization Problem at the beginning of Section \ref{6}), and a conjecture (Conjecture \ref{6.7}). In Section \ref{2} we introduce sets of distances and unions of sets of lengths. We provide a full and self-contained proof for the Structure Theorem for Unions of Sets of Lengths (Theorem \ref{2.6}), and outline an argument that finitely generated commutative monoids satisfy all assumptions of that Structure Theorem. In Section \ref{3} we discuss commutative Krull monoids. This class includes Dedekind domains and hence rings of integers of algebraic number fields, and we will provide an extended list of examples  stemming from a variety of mathematical areas. The central strategy for studying sets of lengths in a given class of semigroups is to construct homomorphisms (called transfer homomorphisms) which can be used to transfer analogous results in a simpler class of semigroups directly back to the more complex class. In Section \ref{4} we discuss transfer homomorphisms, show that they preserve sets of lengths, and provide a self-contained proof for the  fact that there is a transfer homomorphism from commutative Krull monoids to monoids of zero-sum sequences (which are  monoids having a combinatorial flavor that  will often be the simpler class of semigroups to which the more complex semigroup is reduced). We provide an extended list of  transfer Krull monoids (these are the monoids allowing a transfer homomorphism to a monoid of zero-sum sequences), and then we restrict our discussion to this class of monoids. In Section \ref{5} we discuss the Structure Theorem for Sets of Lengths and provide examples showing that all aspects addressed in the Structure Theorem occur naturally. In  Section \ref{6} we discuss sets of lengths of the monoid of zero-sum sequences over a finite abelian group  (the transfer machinery of Section \ref{4} guarantees that these sets of lengths coincide with the sets of lengths of a ring of integers). They can be studied with methods from Additive Combinatorics and their structure is by far the best understood (among all  classes of monoids). In this setting, unions of sets of lengths and the set of distances are intervals and they have natural upper bounds (Proposition \ref{6.1}). In spite of the fact that almost all (in a certain sense) sets of lengths are intervals (Theorem \ref{6.3}) we  conjecture  that the system of sets of lengths is a characteristic for the group  (Conjecture \ref{6.7}). In order to keep this article as self-contained as possible, we do not mention
 arithmetical concepts beyond sets of lengths (such as catenary and tame degrees), or factorization theory in rings with zero-divisors, or divisibility theory in non-atomic rings.

\section{Basic Notation and Unions of Sets of Lengths} \label{2}

We denote by $\mathbb N$ the set of positive integers and set $\N_0 = \N \cup \{0\}$. For integers $a, b \in \Z$, we denote by $[a, b] = \{ x \in \Z \mid a \le x \le b \}$ the (discrete) interval between $a$ and $b$, and by an interval we always mean a  set of this form.
Let $L, L' \subset \Z$ be subsets of the integers. Then $L +L' = \{ a+b \mid a \in L, b \in L' \}$ is the {\it sumset} of $L$ and $L'$. Thus we have $L + \emptyset = \emptyset$, and we set  $-L = \{ -a \mid a \in L\}$. For an integer $m \in \Z$, $m+L = \{m\}+L$ is the shift of $L$ by $m$. For $k \in \N$, we denote by $k L = L + \ldots + L$ the {\it $k$-fold sumset} of $L$ and by $k \cdot L = \{ k a \mid a \in L \}$ the {\it dilation} of $L$ by $k$.
A positive integer $d \in \N$ is called a {\it distance} of $L$ if there are $k, \ell \in L$ with $\ell - k = d$ and the interval $[k, \ell]$ contains no further elements of $L$.   We denote by  $\Delta (L) \subset \N$ the {\it set of distances} of $L$. By definition, we have $\Delta (L) = \emptyset$ if and only if $|L| \le 1$, and $L$ is an arithmetical progression if and only if $|\Delta (L)| \le 1$ . For $L \subset \N$, we denote by $\rho (L) = \sup L/ \min L \in \Q_{\ge 1} \cup \{\infty\}$ the {\it elasticity} of $L$ and we set $\rho ( \{0\}) = 1$.

By a {\it semigroup}, we always mean an associative semigroup, and if not stated otherwise, we use multiplicative notation. Let $S$ be a semigroup. We say that $S$ is cancelative if for all elements $a, b, c \in S$, the equation $ab = ac$ implies $b=c$ and the equation $ba = ca$ implies $b=c$.
All rings and semigroups are supposed to have an identity, and all ring and semigroup homomorphisms preserve the identity.
By a {\it monoid}, we mean a cancelative semigroup. Clearly, subsemigroups of groups are monoids, and finite monoids are groups. If $aS \cap bS \ne \emptyset$ and $Sa \cap Sb \ne \emptyset$ for all $a, b \in S$, then $S$ has  a (unique left and right) quotient group which will be denoted by $\mathsf q (S)$.  If $R$ is a ring, then the set of cancelative elements $R^{\bullet}$ is a monoid. A {\it domain} $D$  is a ring in which zero is the only zero-divisor (i.e., $D^{\bullet} = D\setminus \{0\}$).
We use the abbreviation ACC for the ascending chain condition on ideals.

\medskip
 Let $P$ be a set. We denote by $\mathcal F^* (P)$ the {\it free monoid} with basis $P$, the elements of which may be viewed as words on the alphabet $P$. We denote by $|\cdot| \colon \mathcal F^* (P) \to \N_0$ the function which maps each word onto its length.
The {\it free abelian monoid} with basis $P$ will be denoted by $\mathcal F (P)$. Every $a \in \mathcal F (P)$ has a unique representation of the form
\[
a = \prod_{p \in P}p^{\nu_p} \,, \quad \text{where} \ \nu_p \in \N_0 \ \text{and} \ \nu_p=0 \ \text{for almost all} \ p \in P \,.
\]
Therefore, for every $p \in P$, there is a homomorphism (called the {\it $p$-adic exponent}) $\mathsf v_p \colon \mathcal F (P) \to \N_0$ defined by $\mathsf v_p (a) = \nu_p$. Similar to the case of free monoids, we denote by $|\cdot| \colon \mathcal F (P) \to \N_0$ the usual length function, and we observe that $|a|=\sum_{p \in P} \mathsf v_p (a)$ for all $a \in \mathcal F (P)$.

Let $H$ be a monoid and let $a, b \in H$. The element $a$ is said to be {\it invertible} if there exists an element $a' \in H$ such that $a a' = a' a = 1$. The set of invertible elements of $H$ will be denoted by $H^{\times}$, and we say that $H$ is reduced if $H^{\times} = \{1\}$.

The element $a \in H$ is called {\it irreducible} (or an {\it atom}) if $a \notin H^{\times}$ and, for all $u, v \in H$, $a = u v$ implies that $u \in H^{\times}$ or $v \in H^{\times}$. The monoid $H$ is said to be {\it atomic} if every $a \in H \setminus H^{\times}$ is a product of finitely many atoms of $H$.
While most integral domains introduced in elementary courses are atomic, not all such algebraic objects are.  An elementary example of a non-atomic monoid can be found in \cite[p. 166]{Ch14a}.
If $a \in H$ and $a = u_1 \cdot \ldots \cdot
u_k$, where $k \in \mathbb N$ and $u_1, \ldots, u_k \in \mathcal A
(H)$, then we say that $k$ is the {\it length} of the factorization.
For $a \in H \setminus H^{\times}$, we call
\[
\mathsf L_H (a) = \mathsf L (a) = \{ k \in \mathbb N \mid a \
\text{has a factorization of length} \ k \} \subset \N
\]
the {\it set of lengths} of $a$. For convenience, we set $\mathsf L
(a) = \{0\}$ for all $a \in H^{\times}$. By definition, $H$ is
atomic if and only if $\mathsf L (a) \ne \emptyset$ for all $a \in
H$. Furthermore, it is clear that the following conditions are equivalent: (1) $\mathsf L (a) = \{1\}$; (2) $a \in \mathcal A (H)$; (3) $1 \in \mathsf L (a)$. If $a, b \in H$, then $\mathsf L (a) + \mathsf L (b) \subset \mathsf L (a b)$.
If $H$ is commutative, then $H_{\red} = H/H^{\times} = \{ aH^{\times} \mid a \in H \}$ is the associated reduced monoid, and $H$ is called {\it factorial} if $H_{\red}$ is free abelian.
We say that $H$ is a \BF-{\it monoid} (or a bounded
factorization monoid) if $\mathsf L (a)$ is finite and nonempty for
all $a \in H$. We call
\[
\mathcal L (H) = \{ \mathsf L (a) \mid a \in H \}
\]
the {\it system of sets of lengths} of $H$. So if $H$ is a
\BF-monoid, then $\mathcal L (H)$ is a set of finite nonempty
subsets of the non-negative integers. We say that $H$ is {\it half-factorial} if $|L|=1$ for every $L \in \mathcal L (H)$. Clearly, factorial monoids are half-factorial, and in 1960 Carlitz showed that the ring of integers of an algebraic number field is half-factorial if and only if the class group has at most two elements (see Propositions \ref{3.3} and \ref{4.3}). Since then,
half-factoriality has been a central topic in factorization theory (see, for example,  \cite{Co-Sm11a, Ge-Ka-Re15a, Ma-Ok16a}).
We focus in this paper on the structure of sets of lengths in non-half-factorial BF-monoids. We start with two simple lemmas.

\begin{lemma} \label{2.1}
Let $H$ be a monoid.
\begin{enumerate}
\item If $H$ satisfies the {\rm ACC} on principal left ideals and the {\rm ACC} on principal right ideals, then $H$ is atomic.

\item Suppose that $H$ is atomic. Then $H$ is either half-factorial or for every $N \in \N$ there is an element $a_N \in H$ such that $|\mathsf L (a_N)| > N$.
\end{enumerate}
\end{lemma}

\begin{proof}
For ease of discussion, suppose that $aH=Ha$ for all $a \in H$ (see \cite[Proposition 3.1]{Sm13a} for details in the general case). Assume to the contrary that the set $\Omega$ of all nonunits $a \in H$ that are not products of atoms is nonempty. If $a \in \Omega$, then $a=bc$ with nonunits $b,c \in H$, and either $b \in \Omega$ or $c \in \Omega$. Thus for any $a \in H$, there is some $a' \in \Omega$ with $aH \subsetneq a'H$. Starting from an arbitrary $a \in H$, this gives rise to a properly ascending chain of principal ideals, a contradiction.

To verify the second statement,
suppose that $H$ is atomic but not half-factorial. Then there exist an element $a \in H$, integers $k, \ell \in \N$ with $k < \ell$, and atoms $u_1, \ldots, u_k, v_1 \ldots, v_{\ell} \in \mathcal A (H)$ such that $a = u_1 \cdot \ldots \cdot u_k = v_1 \cdot \ldots \cdot v_{\ell}$.  Then for every $N \in \N$ we have
\[
a^N = (u_1 \cdot \ldots \cdot u_k)^{\nu}(v_1 \cdot \ldots \cdot v_{\ell})^{N-\nu} \quad \text{for all} \quad \nu \in [0,N] \,,
\]
and hence $\{ \ell N - \nu (\ell -k) \mid \nu \in [0,N] \} \subset \mathsf L (a^N)$.
\end{proof}

Let $H$ be a monoid. A function $\lambda \colon H \to \N_0$ is called a {\it right length function} (resp. a left length function) if $\lambda (a) < \lambda (b)$ for all $a \in H$ and all $b \in aH  \setminus aH^{\times}$ (resp. all $b \in Ha \setminus H^{\times}a$).

\begin{lemma} \label{2.2}
Let $H$ be a monoid.
\begin{enumerate}
\item $H$ is a \BF-monoid if and only if there is a left length function (or a right length function)  $\lambda \colon H \to \N_0$.

\item If $H$ is a \BF-monoid, then $H$ satisfies the {\rm ACC} on principal left ideals and on principal right ideals.

\item Submonoids of free monoids and of free abelian monoids are \BF-monoids.
\end{enumerate}
\end{lemma}

\begin{proof}
1. If $H$ is a BF-monoid, then the function $\lambda \colon H \to \N_0$, defined by $a \mapsto \max \mathsf L (a)$ for every $a \in H$, is a right length function and a left length function. Now, if there is a right length function $\lambda \colon H \to N_0$, then we have to show that $H$ is a \BF-monoid (the proof is completely analogous for left length functions).  First we observe  that, if $b \in H^{\times}$ and $c \in H \setminus H^{\times} = bH
\setminus bH^{\times}$, then  $\lambda (c)
> \lambda (b) \ge 0$. We claim that every $a \in H \setminus
H^{\times}$ can be written as a product of finitely many atoms, and that $\sup
\mathsf L (a) \le \lambda (a)$. If $a \in \mathcal A (H)$, then
$\mathsf L (a) = \{1\}$, and the assertion holds. Suppose that $a
\in H$ is neither an atom nor a unit. Then $a$ has a product
decomposition of the form
\[
a = u_1 \cdot \ldots \cdot u_k \quad \text{ where} \quad  k \ge 2 \
\text{ and}  \ u_1, \ldots, u_k \in H \setminus H^{\times} \,.
\]
For $i \in [0, k]$, we set $a_i = u_1 \cdot \ldots \cdot u_i$, with $a_0=1$, and
hence $a_{i+1} \in a_iH \setminus a_iH^{\times}$ for all $i \in [0,
k-1]$. This implies that $\lambda (a) = \lambda (a_k) > \lambda
(a_{k-1}) > \ldots > \lambda (a_1) > 0$ and thus $\lambda (a) \ge
k$. Therefore there is a maximal $k \in \N$  such that $a = u_1
\cdot \ldots \cdot u_k$ where $u_1, \ldots, u_k \in H \setminus
H^{\times}$, and this implies that $u_1, \ldots, u_k \in \mathcal A
(H)$ and $k = \max \mathsf L (a) \le \lambda (a)$.

2. Suppose that $H$ is a \BF-monoid. Let $\lambda \colon H \to \N_0$ be a right length function and assume to the contrary that there is a properly ascending chain of principal right ideals $a_0H \subsetneq a_1H \subsetneq a_2H \subsetneq \ldots $. Then $\lambda (a_0) > \lambda (a_1) > \lambda (a_2) > \ldots$, a contradiction. Similarly, we can show that $H$ satisfies the  ACC on principal left ideals.

3. Use 1., and note that the restriction of a length function is a length function.
\end{proof}

Next  we introduce a main parameter describing the structure of sets of lengths, namely
\[
\Delta (H) = \bigcup_{L \in \mathcal L (H)} \Delta (L)
\]
which is the {\it set of distances} of $H$ (also called the {\it delta set} of $H$).  We open by showing that $\Delta(H)$ satisfies a fundamental property.

\begin{proposition} \label{2.3}
Let $H$ be an atomic monoid with $\Delta (H) \ne \emptyset$. Then $\min \Delta (H) = \gcd \Delta (H)$.
\end{proposition}

\begin{proof}
We set $d = \gcd \Delta (H)$.   Clearly, it suffices to show that $d
\in \Delta (H)$. There are $t \in \N, d_1, \ldots, d_t \in \Delta
(H)$ and $m_1, \ldots, m_t \in \Z \setminus \{0\}$ such that $d =
m_1 d_1 + \ldots + m_t d_t$. After renumbering if necessary, there is
some $s \in [1, t]$ such that $m_1, \ldots, m_s$, $-m_{s+1}, \ldots,
-m_t$ are positive. For every $\nu \in [1,t]$, there are $x_{\nu} \in \N$
and $a_{\nu} \in H$ such that
$
\{ x_{\nu}, x_{\nu} + d_{\nu} \} \subset \mathsf L (a_{\nu}) \ \text{for every}
\ \nu \in [1,s]
\quad \text{
and}
\quad
\{ x_{\nu}-d_{\nu}, x_{\nu} \} \subset \mathsf L (a_{\nu}) \ \text{for every}
\ \nu \in [s+1,t]$.
This implies that
\[
\begin{aligned}
\{ m_{\nu}x_{\nu}, m_{\nu}x_{\nu} + m_{\nu}d_{\nu} \} & \subset \mathsf L (a_{\nu}^{m_{\nu}}) \ \text{for every}
\ \nu \in [1,s]
\quad \text{
and} \\
\{ -m_{\nu}x_{\nu}+m_{\nu}d_{\nu}, -m_{\nu}x_{\nu} \} & \subset \mathsf L (a_{\nu}^{-m_{\nu}}) \ \text{for every}
\ \nu \in [s+1,t] \,.
\end{aligned}
\]
We set $a = a_1^{m_1} \cdot \ldots \cdot a_s^{m_s}a_{s+1}^{-m_{s+1}} \cdot \ldots \cdot a_t^{-m_t}$, and observe that $\mathsf L (a) \supset $
\[
\Bigl\{ k  := \sum_{\nu=1}^s m_{\nu} x_{\nu} - \sum_{\nu=s+1}^t m_{\nu} x_{\nu}, \ \ell :=
\sum_{\nu=1}^s m_{\nu}(x_{\nu}+d_{\nu}) - \sum_{\nu=s+1}^t m_{\nu} (x_{\nu}-d_{\nu}) \Bigr\} \,.
\]
Since $\ell - k = d \le \min \Delta (H)$, it follows that  $d \in \Delta ( \mathsf L (a)) \subset \Delta (H)$.
\end{proof}

We now  introduce unions of sets of lengths. Let $H$ be an atomic monoid, and $k, \ell \in \N$. In the extremal case where $H = H^{\times}$ it is convenient to  set $\mathcal U_k (H) = \{k\}$. Now suppose that $H \ne H^{\times}$. We define $\mathcal U_k (H)$ to be the set of all $\ell \in \N$ such that we have an equation of the form
\[
u_1 \cdot \ldots \cdot u_k = v_1 \cdot \ldots \cdot v_{\ell} \quad \text{where} \quad u_1, \ldots, u_k, v_1, \ldots , v_{\ell} \in \mathcal A (H) \,.
\]
In other words, $\mathcal U_k (H)$ is the {\it union of sets of lengths} containing $k$. Clearly, we have $k \in \mathcal U_k (H)$, and $1 \in \mathcal U_k (H)$ if and only if $k=1$, and $\mathcal U_1 (H) = \{1\}$.  Furthermore, we have $k \in \mathcal U_{\ell} (H)$ if and only if $\ell \in \mathcal U_k (H)$. Now we define
\[
\rho_k (H) = \sup \mathcal U_k (H) \quad \text{and} \quad \lambda_k (H) = \min \mathcal U_k (H) \,,
\]
and we call $\rho_k (H)$ the {\it $k$th elasticity} of $H$. Since $\mathcal U_k (H) + \mathcal U_{\ell} (H) \subset \mathcal U_{k+\ell}(H)$, we infer that
 \begin{equation} \label{basic-inequ}
      \lambda_{k+\ell} (H) \le \lambda_k (H) + \lambda_{\ell} (H) \le k+\ell \le \rho_k
      (H) + \rho_{\ell} (H) \le \rho_{k+\ell} (H) \,,
      \end{equation}
and we will use these inequalities without further mention. The invariant
\[
\rho (H) = \sup \{ \rho (L) \mid L \in \mathcal L (H) \} \in \R_{\ge 1} \cup \{\infty\}
\]
is called the {\it elasticity} of $H$. If $k \in \N$ and $L \in \mathcal L (H)$ with $\min L \le k$, then $\sup L \le \rho (H) \min L \le k \rho (H)$ and hence $\rho_k (H) \le k \rho (H)$. Thus if the elasticity $\rho (H)$ is finite, then $H$ is a BF-monoid and $\rho_k (H) < \infty$ for all $k \in \N$.

The $k$th elasticities  $\rho_k (H)$ were first studied for rings of integers of algebraic number fields  and then in the setting of abstract semigroups \cite{Ge-Le90}.
Unions of sets of lengths have been introduced by Chapman and Smith \cite{Ch-Sm90a} in the setting of commutative Dedekind domains, and
the Structure Theorem for Unions of Sets of Lengths (as given in Theorem \ref{2.6}) has first been proved  in \cite{Ga-Ge09b} (in a commutative setting).

\begin{proposition} \label{2.4}
Let $H$ be an atomic monoid with $H \ne H^{\times}$.
\begin{enumerate}
\item For every $k \in \N$ we have
      $\rho_k (H) = \sup \{ \sup L \mid  L \in \mathcal L (H), \ \min L \le k \} \ge \sup \{ \sup L \mid  L \in \mathcal L (H), \ k = \min L \}$, \
      and equality holds if  $\rho_k (H) < \infty$.

\item
      \[
      \begin{aligned}
      \rho (H) & = \sup \Bigl\{ \frac{\rho_{k} (H)}{k} \; \Bigm| \; k \in
      \N \Bigr\} = \lim_{k \to \infty}\frac{\rho_k(H)}{k} \quad \text{and} \\ \frac{1}{\rho (H)}&  =  \inf \Bigl\{ \frac{\lambda_{k} (H)}{k} \; \Bigm| \; k \in
      \N \Bigr\}  = \lim_{k \to \infty}\frac{\lambda_k(H)}{k}   \,.
      \end{aligned}
      \]

\item Suppose that $\rho (H) < \infty$. Then the following statements are equivalent.
      \begin{enumerate}
      \item There is an $L \in \mathcal L (H)$ such that $\rho (L)=\rho (H)$.
      \item There is an $N \in \N$ such that $kN \rho (H) = \rho_{kN} (H)$ for all $k \in \N$.
      \item There is some $k \in \N$ such that $k \rho (H) = \rho_k (H)$.
      \end{enumerate}
      If one of the above statements holds, then there is some $M \in \N$ such that $\rho_k (H) - \rho_{k-1} (H) \le M$ for all $k \ge 2$.
\end{enumerate}
\end{proposition}

\begin{proof}
1. Let $k \in \N$. We define $
\rho_k' (H) = \sup \{ \sup L \mid L \in \mathcal L (H), \ \min L \le k \}$,
$\rho_k'' (H) = \sup \{ \sup L \mid L \in \mathcal L (H), \ k = \min L \}$,
and  obtain by definition $\rho_k' (H) \ge \rho_k (H) \ge \rho_k''(H)$. Hence we must prove that $\rho_k (H) \ge \rho_k'
(H)$, and if $\rho_k(H) < \infty$, then also $\rho_k ''(H) \ge
\rho_k' (H)$. Let $a \in H$ with  $\min \mathsf L (a) = \ell \le k$ and $u
\in \mathcal A (H)$. Then $k \in \mathsf L (a u^{k - \ell})$ implies
$\rho_k (H) \ge \sup \mathsf L (a u^{k - \ell}) \ge \sup \mathsf L
(a) + k - \ell \ge \sup \mathsf L (a)$, and therefore $\rho_k (H)
\ge \rho_k' (H)$.
Assume now that $\rho_k (H) < \infty$, and let $a$ be as above
such that  $\sup \mathsf L (a) = \rho_k' (H)$. Then  $\min \mathsf L
(a u^{k - \ell}) \le \min \mathsf L (a) + (k - \ell) = k$ and therefore
\[
\rho_k' (H) \ge \sup \mathsf L (a u^{k - \ell}) \ge \sup \mathsf L (a) + (k - \ell) \ge \rho_k' (H)\,.
\]
Thus $k = \ell = \min \mathsf L (a)$ and $\rho_k'' (H) \ge \sup \mathsf L (a) = \rho_k' (H)$.

2.  If there is a $k \in \N$ such that $\rho_k (H) = \infty$, then all three terms of the first equation are equal to infinity. Suppose   that $\rho_k (H) < \infty$ for all $k \in \N$.
Then the first equality follows from 1. To verify the first limit
assertion, let $\rho' < \rho (H)$ be arbitrary. We must prove that
$\rho_k(H) \ge k \rho'$ for all sufficiently large $k$. If $\rho'
< \rho '' < \rho(H)$, then there exists some $q_0 \in \N$ such
that
\[
\frac{q\rho'' +1}{q+1} > \rho' \quad \text{for all} \quad q \ge q_0 \,,
\]
and there exists some $N \in \N$ such that $\rho_N(H)  > N \rho''$.
If $k > Nq_0$, then $k = Nq + r$ for some $q \ge q_0$ and $r \in [1,N]$, and $\rho_k(H) \ge q \rho_N(H) + \rho_r(H) \ge q
\rho_N(H)  +r$ by Inequality \eqref{basic-inequ}. Since $\rho_N(H) \ge N$, it follows that
\[
\frac{\rho_k(H)}{k} \ge \frac{q\rho_N(H) + r}{qN+r} \ge \frac{q\rho_N(H) + N}{qN + N} > \frac{qN\rho''+ N}{qN + N} > \rho' \,.
\]
Since $
\rho (H) = \sup \Big\{ \frac{m}{n} \; \Bigm| \; m,n \in L , \{0\} \ne L \in \mathcal L (H) \Big\}$,
we have
\[
\frac{1}{\rho (H)} = \inf \Big\{ \frac{m}{n} \; \Bigm| \; m,n \in L , \{0\} \ne L \in \mathcal L (H) \Big\}
\]
with $1/\rho (H)=0$ if $\rho (H)= \infty$. The verification of the second limit assertion runs along the same lines as the proof of the first one (just replace $\rho_k (H)$ by $\lambda_k (H)$ and reverse all inequality signs).

3. In order to show the implication (a) $\Rightarrow$ (b), let $a \in H$ with $L = \mathsf L (a) \in \mathcal L (H)$ such that $\rho (L)=\rho (H)$ and set $N = \min L$. Then $kN \in \mathsf L (a^k)$ for all $k \in \N$, and thus
\[
\frac{\rho_{k N} (H)}{k N} \ge \frac{\sup \mathsf L (a^{k})}{k N} \ge \frac{k \ \sup \ \mathsf L(a)}{k N} = \rho (L) = \rho (H) \ge \frac{\rho_{k N} (H)}{k N}.
\]
The implication (b) $\Rightarrow (c)$ is obvious. Suppose that (c) holds and let $k \in \N$ such that $k \rho (H) = \rho_k (H)$. Let $L \in \mathcal L (H)$ such that $\min L = k$ and $\rho_k (H) = \max L$. Then $k \rho (H) = \max L = k \rho (L)$ and hence $\rho (L) = \rho (H)$.

Suppose that the equivalent statements hold and let $N \in \N$ such that $mN \rho (H) = \rho_{mN} (H)$ for all $m \in \N$. Let $k \in \N$ be given and set $k=iN+j$ with $i \in \N_0$ and $j \in [1,N]$. Then we infer that $iN\rho (H) + j \le \rho_{iN}(H)+\rho_j(H) \le \rho_{iN+j}(H)$
and
\[
(iN+j)\rho (H) -N(\rho(H)-1) \le (iN+j)\rho (H) -j (\rho (H)-1) \le \rho_{iN+j}(H) \,.
\]
Thus $k \rho (H) - \rho_k (H) \le M' := N(\rho(H)-1)$ for all $k \in \N$. If $k \ge 2$, then $\rho_k (H) \le k \rho (H)$, $(k-1)\rho (H) - \rho_{k-1}(H) \le M'$, and hence $\rho_k (H) - \rho_{k-1} (H) \le M'+\rho (H)$.
\end{proof}

\medskip
Unions of sets of lengths do have -- in a variety of settings -- a structure which is only slightly more general than that of arithmetical progressions. In order to show this, we introduce the concept of almost arithmetical progressions.

\begin{definition} \label{2.5}
Let \ $d \in \mathbb N$ \ and \ $M \in \mathbb N_0$. A subset \ $L
\subset \mathbb Z$ \ is called an {\it almost arithmetical
progression} ({\rm AAP} for short) with {\it difference} $d$ and
{\it bound} $M$ if
\[
L = y + (L' \cup L^* \cup L'') \subset y + d \mathbb Z
\]
where $y\in \Z$ and $L^*$ is a nonempty (finite or infinite) arithmetical progression with difference
$d$  such that $\min L^* = 0$, $L' \subset [-M,-1]$, $L'' \subset
\sup L^* + [1, M]$ (with the convention that $L'' = \emptyset$ if
$L^*$ is infinite).
\end{definition}

Clearly, every single finite set is an AAP with a particularly trivial choice of $M$ and $d$. Suppose we have an atomic monoid with nonempty set of distances. Then, by Lemma \ref{2.1}, sets of lengths become arbitrarily large, whence the unions $\mathcal U_k (H)$ are growing as $k$ is growing. The next theorem states  (under the given assumptions) that all unions $\mathcal U_k (H)$ are AAPs with the critical point  that a single choice of $M$ works for all sufficiently large $k$. Once this $M$ is chosen, it says that all unions have a "middle" piece that becomes larger and larger as $k$ grows and that "middle" part has a very rigid structure. The poorly behaved "end" pieces are bounded in size by a constant and in their distance from the middle.

\begin{theorem}[{\bf Structure Theorem for Unions of Sets of Lengths}] \label{2.6}~
Let  $H$  be an atomic monoid with finite nonempty set of
distances  $\Delta (H)$ and $d = \min \Delta (H)$. Suppose that
either,  $\rho_k (H) = \infty$ \ for some $k \in \mathbb N$, or
that there is an $M \in \mathbb N$ such that \ $\rho_k (H) -
\rho_{k-1} (H) \le M$ \ for all $k \ge 2$.
\begin{enumerate}
\item There exist constants \ $k^*$ \ and \ $M^* \in \mathbb N$ \ such
      that  for all $k \ge k^*$, \ $\mathcal U_k (H)$ \ is an {\rm
      AAP} with difference $d$ and bound $M^*$. Moreover, if \
      $\rho_k (H) < \infty$ \ for all $k \in \N$, then the assertion
      holds for \ $k^* = 1$.

\item We have
      \[
      \lim_{k \to \infty} \frac{|\mathcal U_k (H)|}{k} = \frac{1}{d}
      \Bigl( \rho (H) - \frac{1}{\rho (H)} \Bigr) \,.
      \]
\end{enumerate}
\end{theorem}

\begin{proof}
1.  For every $k \in \N$, we have $\lambda_k (H), k \in \mathcal U_k (H)$, and since every $d' \in \Delta (H)$ is a multiple of $d$ by Proposition \ref{2.3}, it follows that $\mathcal U_k (H) \subset \lambda_k (H) + d \mathbb N_0$. Thus it
remains to show that there exist constants $k^*, M^* \in \mathbb N$
such that, for all $k \ge k^*$,
\[
\mathcal U_k (H) \cap [k, \rho_k (H) - M^*] \quad \text{and} \quad
\mathcal U_k (H) \cap [\lambda_k (H) +M^*, k] \tag{$*$}
\]
are arithmetical progressions with difference $d$ (note, if $\rho_k (H) = \infty$, then $[k, \rho_k (H) - M^*] = \N_{\ge k}$). We break the proof into three steps.

1.(a) First, we show that the left set in ($*$) is an arithmetical progression. Since $d \in \Delta (H)$, there is an element $a \in H$ and $m
\in \mathbb N$ such that $\{m, m+d\} \subset \mathsf L (a)$. Since $\Delta (H)$ is finite and
$\min \Delta (H) = \gcd \Delta (H)$ by Proposition \ref{2.3}, $\psi := \rho \bigl( \Delta (H) \bigr)  \in
\mathbb N$. Then the $\psi$-fold sumset satisfies the containment
\[
U^* := \psi \{m, m+d\} = \{\psi m  , \psi m + d, \ldots, \psi m + \psi d \}
\subset \mathsf L (a^{\psi}) \,.
\]
We set  $k_0 = \psi m$ and observe that $U^*   \subset \mathcal U_{k_0} (H)$, say
$\mathcal U_{k_0} (H) = U' \uplus U^* \uplus U''$,
where $U' = \{ x \in \mathcal U_{k_0} (H) \mid x < k_0 \}$ and $U'' = \{ x \in \mathcal U_{k_0} (H) \mid x > k_0+\psi d\}$.
Let $k^* \in \N$ with  $k^* \ge  2k_0$. If there is some $\ell \in
\mathbb N$ with $\rho_{\ell} (H) = \infty$, then let $\ell_0$ denote the smallest
such $\ell \in \mathbb N$, and we suppose further that $k^* - k_0 \ge
\ell_0$. Now let $k \ge k^*$ be given. Then \ $\mathcal U_{k_0} (H) + \mathcal U_{k-k_0} (H)  =$
\[
 \Bigl( U' +
\mathcal U_{k-k_0} (H) \Bigr) \cup \Bigl( U^* + \mathcal U_{k-k_0}
(H) \Bigr) \cup \Bigl( U'' + \mathcal U_{k-k_0} (H) \Bigr)
  \subset \mathcal U_k (H) \,.
\]
Clearly, we have $k \in U^* + \mathcal U_{k-k_0} (H)$. Since \ $\max
\Delta \bigl( \mathcal U_{k-k_0} (H) \bigr) \le \max \Delta (H)$ \
and \ $\Delta \bigl( \mathcal U_{k-k_0} (H) \bigr) \subset d \N$, it follows that $U^* + \mathcal
U_{k-k_0} (H)$ is an arithmetical progression with difference $d$.
If there is some $\ell \in \mathbb N$ such that $\rho_{\ell} (H) = \infty$,
then \ $\rho_{k-k_0} (H) = \rho_k (H) = \infty$ \ and
\[
\bigl( U^* + \mathcal U_{k-k_0} (H) \bigr) \cap \mathbb N_{\ge k} =
k + d \mathbb N_0 = \mathcal U_k (H) \cap \mathbb N_{\ge k} \,.
\]
Suppose that $\rho_{\ell} (H) < \infty $ for all $\ell \in \mathbb N$. Then
\[
\max \mathcal U_k (H) - \max \Bigl( U^* + \mathcal U_{k-k_0} (H)
\Bigr) = \rho_k (H) - \max U^* - \rho_{k-k_0} (H) \le k_0 M \,,
\]
and hence
\[
\bigl( U^* + \mathcal U_{k-k_0} (H) \bigr) \cap [k, \rho_k (H) - k_0
M] = \mathcal U_k (H) \cap [k, \rho_k (H) - k_0 M]
\]
is an arith. progression with difference $d$. Thus the assertion follows with $M^* = k_0 M$.

1.(b). By 1.(a), there are $k^*, M^* \in \mathbb N$ such that for
all $k \ge k^*$, the set $\mathcal U_k (H) \cap [k, \rho_k (H) -
M^*]$ is an arithmetical progression with difference $d$. Without
restriction we may suppose that $M^* \ge k^*$.

Let $k \ge k^*$ and $\ell = \lambda_k (H)$. We show that $\mathcal U_k
(H) \cap [\ell+M^*, k]$ is an arithmetical progression with difference
$d$. Let $m \in [\ell+M^*, k]$ be such that $k-m$ is a multiple of $d$. In
order to show that $m \in \mathcal U_k (H)$, we verify that $k \in
\mathcal U_m (H)$. Since
\[
k \le \rho_{\ell} (H) \quad \text{and} \quad \ell + M^* \le m \,,
\]
it follows that $k \le \rho_{\ell} (H) \le \rho_{m - M^*} (H)$ and hence
\[
k + M^* \le \rho_{m - M^*} (H) + M^* \le \rho_{m - M^*} (H) +
\rho_{M^*} (H) \le \rho_m (H) \,.
\]
Since $k \in m + d \mathbb N_0$ with $k \le \rho_m (H) - M^*$ and
$\mathcal U_m (H) \cap [m, \rho_m (H) - M^*]$ is an arithmetical
progression with difference $d$, it
follows that $k \in \mathcal U_m (H)$.

1.(c). Suppose that \ $\rho_k (H) < \infty$ \ for all  $k \in \N$. In this  step, we show that for all $k \in \N$ the sets  $\mathcal U_k (H)$ are AAPs with difference $d$ and some bound $\widetilde M$. Suppose that  the assertion holds with the constants $k^*$ and $M^* \in
\N$. Since for all \ $k \in [1, k^*-1]$,
\[
\mathcal U_k (H) = \bigl( [\lambda_k (H), k-1] \cap \mathcal U_k (H)
\bigr) \cup \{k\} \cup \bigl( [k+1, \rho_k (H)] \cap \mathcal U_k
(H) \bigr)
\]
is an AAP with bound $M' = \max \{ k - \lambda_k (H), \rho_k (H) - k
\mid k \in [1, k^*-1] \}$, it follows that for all $k \in \N$ the
sets \ $\mathcal U_k (H)$ \ are AAPs with difference $d$ and bound
$\widetilde M = \max \{M^*, M' \}$.

2. If there is some \ $k \in \mathbb N$ \ such that \ $\rho_k (H) =
\infty$, then both the left and the right hand side of the asserted
equation are infinite. Suppose that
$\rho_k (H) < \infty$ \ for all \ $k \in \mathbb N$. \ By 1. there
are \ $k^* \in \mathbb N$ \ and \ $M^* \in d\N$ \ such that, for all
$k \ge k^*$, \ $\mathcal U_k (H) \cap [\lambda_k (H) + M^*, \rho_k
(H) - M^*]$ \ is an arithmetical progression with difference $d$.
Thus for all \ $k \ge k^*$ \ we obtain that
\[
\frac{\bigl( \rho_k (H) - M^*\bigr) - \bigl(\lambda_k (H) +
M^*\bigr) + d}{dk} \le \frac{|\mathcal U_k (H)|}{k} \le \frac{\rho_k
(H) - \lambda_k (H) + d }{dk} \,.
\]
Since, by Proposition \ref{2.4}.2, $
\lim_{k \to \infty} \frac{\rho_k (H)}{k} = \rho (H) \quad \text{and}
\quad \lim_{k \to \infty} \frac{\lambda_k (H)}{k} = \frac{1}{\rho
(H)}$,
the assertion follows.
\end{proof}

We end this section with a discussion of finitely presented monoids and of commutative finitely generated monoids. Let $H$ be a monoid. For every generating set $P$ of $H$, there is an epimorphism $\psi \colon \mathcal F^* (P) \to H$ and $\mathcal F^* (P)/ \ker ( \psi) \cong H$, where
\[
\ker (\psi) = \{(x,y) \in \mathcal F^* (P) \times \mathcal F^* (P) \mid \psi (x) = \psi (y) \}
\]
is a congruence relation on $\mathcal F^* (P)$. If there is a finite generating set $P$ and a finite set of relations $R \subset \ker (\psi)$ which generates $\ker (\psi)$ as a congruence relation, then $H$ is said to be {\it finitely presented} (by $P$ and $R$). If $R = \{(x_1, y_1), \ldots, (x_t, y_t)\}$, then we write (as usual) $H = \langle P \mid R \rangle = \langle P \mid x_1=y_1, \ldots, x_t=y_t \rangle$.

\begin{proposition} \label{2.7}
If $H = \langle \mathcal A (H) \mid R \rangle$ is a reduced atomic  monoid with a finite set of relations $R$, then the set of distances $\Delta (H)$ is finite.
\end{proposition}

\begin{proof}
We set $A = \mathcal A (H)$, $\psi \colon \mathcal F^* (A) \to H$,  $R = \{(x_1, y_1), \ldots, (x_t, y_t)\}$,  $M = \max \{ \big| |x_1|-|y_1| \big|, \ldots, \big| |x_t|-|y_t| \big| \}$, and  assert that $\Delta (H) \subset [1, M]$. Let $a \in H$. Then $\mathsf L (a) = \{ |x| \mid x \in \mathcal F^* (A) \ \text{with} \ \psi (x)=a \}$. We choose two words $v, w \in \mathcal F^* (A)$ with $\psi (v)=\psi (w)=a$. Since $\mathcal F^* (P)/ \ker ( \psi) \cong H$ and $\ker (\psi)$ is generated by $R$ (as a congruence), there is a sequence of words $v=v_0, \ldots, v_n=w$ in $\psi^{-1} (a) \subset \mathcal F^* (A)$ where $v_{\nu}$ arises from $v_{\nu-1}$ by replacing  $x_i$ by $y_i$ for some $i \in [1,t]$ and all $\nu \in [1, n]$. We set $L = \{ |v_{0}|, \ldots, |v_n|\}$ and obtain    $\Delta (L) \subset [1,M]$. Thus it follows that $\Delta \big( \mathsf L (a) \big) \subset [1,M]$ and hence $\Delta (H) \subset [1,M]$.
\end{proof}

There are atomic finitely presented monoids $H$ such that $\rho (H) = \rho_k (H) = \infty$ for all $k \ge 2$. To provide an example, consider the monoid $H = \langle a, b \mid a^2 = b a^2 b \rangle$ (note that $H$ is an Adyan semigroup and hence cancelative, see \cite[Section 2]{Ba-Sm15} for details). Obviously, $H$ is finitely presented and atomic with $\mathcal A (H) = \{a, b\}$,  and $\rho_2 (H) = \infty$ which implies that $\rho_k (H) = \infty $ for all $k \ge 2$. Since $Ha^2b^i \subsetneq Ha^2b^{i+1}$ for all $i \in \N$, $H$ does not satisfy the ACC on principal left ideals. As another example, the monoid $H = \langle a, b \mid a=bab \rangle$ is a finitely presented monoid which is not atomic (note that $a$ is not a finite product of atoms).
However, this behavior cannot occur in the case of commutative monoids. The next result shows in particular that finitely generated commutative monoids satisfy all assumptions of Theorem \ref{2.6}, and hence they satisfy the Structure Theorem for Unions of Sets of Lengths.

\begin{proposition} \label{2.8}
Let $H$ be a  reduced commutative monoid.
\begin{enumerate}
\item For a subset $A \subset \mathcal A (H)$ the following statements are equivalent.
      \begin{enumerate}
      \item $H$ is atomic and $A = \mathcal A (H)$.

      \item $A$ is the smallest generating set of $H$ (with respect to set  inclusion).

      \item $A$ is a minimal generating set of $H$.
      \end{enumerate}

\item $H$ is finitely generated if and only if $H$ is atomic and $\mathcal A (H)$ is finite.

\item Suppose that $H$ is finitely generated.
      Then $H$ is a \BF-monoid with finite set of distances and finite elasticity. Moreover, there is an $L \in \mathcal L (H)$ such that $\rho (L) = \rho (H)$, and there is an $M \in \N$ such that $\rho_k (H) - \rho_{k-1} (H) \le M$ for all $k \ge 2$.
\end{enumerate}
\end{proposition}

\begin{proof}
1.(a) $\Rightarrow$ (b) Since $H$ is atomic, $A$ is a generating set of $H$ and it remains to show that it is the smallest one. Let $A'$ be any generating set of $H$. If $u \in A$, then $u = v_1 \cdot \ldots \cdot v_k$ with $k \in \N$ and $v_1, \ldots, v_k \in A' \setminus \{1\}$. Since $u$ is an atom, it follows that $k=1$ and $u = v_1 \in A'$. The implication (b) $\Rightarrow (c)$ is obvious.

(c) $\Rightarrow$ (a) It suffices to verify that $A = \mathcal A (H)$. Since $A$ is a minimal generating set, it follows that $1 \notin A$. If $u \in \mathcal A (H)$, then $u = v_1 \cdot \ldots \cdot v_k$ with $k \in \N$ and $v_1, \ldots, v_k \in A$. This implies $k=1$, $u =v_1 \in A$, and thus $\mathcal A (H) \subset A$. Assume to the contrary that there is some $u \in A \setminus \mathcal A (H)$, say $u=vw$ with $v,w \in H \setminus \{1\}$. Then $v = u_1 \cdot \ldots \cdot u_m$ and $w=u_{m+1} \cdot \ldots \cdot u_n$ where $n \in \N_{\ge 2}$, $m \in [1, n-1]$, and $u_1, \ldots, u_n \in A$. Therefore we obtain that $u=u_1 \cdot \ldots \cdot u_n$ and $u \notin \{u_1, \ldots, u_n \}$. Thus $A \setminus \{u\}$ is a generating set of $H$, a contradiction.

2. This assertion follows directly from 1.

3. By 1. and 2., $H$ is atomic and $A = \mathcal A (H)$ is the smallest generating set.
By Redei's Theorem, every finitely generated commutative monoid is finitely presented. Thus $\Delta (H)$ is finite by Proposition \ref{2.7}. Next we show that there is an $L \in \mathcal L (H)$ such that $\rho (H) = \rho (L) < \infty$. This implies  that $H$ is a BF-monoid and  by Proposition \ref{2.4}.3 it follows that there is an $M \in \N$ such that $\rho_k (H) - \rho_{k-1} (H) \le M$ for all $k \ge 2$.

Let $\pi \colon \mathcal F (A)  \to H$ be the canonical epimorphism. We set
\[
S = \{(x,y) \in \mathcal F (A) \times \mathcal F (A) \mid \pi (x) = \pi
(y) \} \quad \text{and} \quad S^* = S \setminus \{(1,1)\}\,,
\]
and we observe that $\rho(H) = \sup \Bigl\{ \frac{|x|}{|y|}\; \Bigm| \;(x,y) \in S^*
\Bigr\}$.
Clearly, it is sufficient to show that this supremum is attained for some pair $(x,y) \in S^*$. There is an isomorphism $f \colon \mathcal F (A) \times \mathcal F (A) \to (\N_0^{A} \times \N_0^{A}, +)$, defined by $(\prod_{u \in A}u^{m_u}, \prod_{u \in A}u^{n_u}) \mapsto \bigl( (m_u)_{u \in A}, (n_u)_{u \in A} \bigr)$. By Dickson's Theorem \cite[Theorem 1.5.3]{Ge-HK06a}, the set $f (S^*)$ has only finitely many minimal points, and let $T \subset S^*$ denote the inverse image of the set of minimal points. Therefore it suffices to prove that
\[
\frac{|x|}{|y|} \le \max \Bigl\{ \frac{|x'|}{|y'|}\; \Bigm| \;
(x',y') \in T \,\Bigr\} \quad \text{for all} \quad (x,y) \in
S^*\,.
\]
We proceed by induction on $|x| + |y|$. If $(x,y) \in T$, then there is
nothing to do. Suppose that  $(x,y) \notin T$.  Then there exist $(x_1, y_1) \in T$ such that $(x,y) = (x_1 x_2,\, y_1
y_2)$ with  $(x_2, y_2) \in \mathcal F (A) \times \mathcal F (A)$. It follows that $(x_2, y_2) \in S^*$, and clearly we have   $|x_j| + |y_j| < |x| + |y|$ for $j \in \{1,2\}$. Then
\[
\frac{|x|}{|y|} = \frac{|x_1|+|x_2|}{|y_1|+|y_2|} < \max \Bigl\{
\frac{|x_1|}{|y_1|},\, \frac{|x_2|}{|y_2|}\Bigr\} \le \max \Bigl\{
\frac{|x'|}{|y'|} \; \Bigm| \; (x',y') \in T \Bigr\}
\]
by the induction hypothesis.
\end{proof}

A most interesting class of finitely generated commutative monoids are numerical monoids.
Their study was initiated by Frobenius in the 19th century and they are still a topic of much research due to their intrinsic relationship with a wide area of mathematics.
A monoid $H$ is said to be {\it numerical} if it is a submonoid of $(\N_0, +)$ such that the complement $\N_0 \setminus H$ is finite. Clearly, numerical monoids are reduced. Let $H$ be a numerical monoid with $H \ne \N_0$. Since $\N_0 \setminus H$ is finite, $H$ has a finite generating set and hence a smallest generating set. Thus Proposition \ref{2.8} implies that the smallest generating set is the set of atoms, and that the elasticity and the set of distances are both finite. Suppose that $\mathcal A (H) = \{n_1, \ldots, , n_t \}$ with $t \in \N$ and $n_1 < \ldots < n_t$. We encourage the reader to check that $\rho (H) = n_t/n_1$ and that $\min \Delta (H) = \gcd \{n_2-n_1, \ldots , n_t - n_{t-1} \}$ (compare with Proposition \ref{6.1}.2). These results were the starting points of detailed investigations of the arithmetic of numerical monoids initiated  by Chapman and Garc{\'i}a-S{\'a}nchez.

Clearly, there are natural connections between the arithmetical invariants of factorization theory and the presentations of a monoid.  This point has been emphasized by Garc{\'i}a-S{\'a}nchez and it
opened the way to an algorithmic approach towards the computational determination of arithmetical invariants. Many algorithms have been implemented in GAP (see the GAP Package \cite{numericalsgps} and  a  survey by Garc{\'i}a-S{\'a}nchez \cite{GS16a}).

\section{Commutative Krull Monoids} \label{3}

It was the observation of the mathematicians of the 19th century that a ring of integers in an algebraic number field need not be factorial (in other words, it need not satisfy the Fundamental Theorem of Arithmetic). This led to the development of ideals (every nonzero ideal in a ring of integers is a unique product of prime ideals whence the Fundamental Theorem of Arithmetic holds for ideals) and subsequently to the development of "divisor theories" from the elements to the ideals. A divisor theory is a divisibility preserving homomorphism to an object which fulfills the Fundamental Theorem of Arithmetic.  Semigroups allowing a divisor theory are now called Krull monoids.

Commutative Krull monoids can be studied with divisor theoretic and with ideal theoretic tools. We start with divisor theoretic concepts. Let $H$ and $D$ be commutative  monoids. A monoid homomorphism $\varphi \colon H \to D$ is said to be:
\begin{itemize}
\item a {\it divisor homomorphism} if $a, b \in H$ and $\varphi (a) \t \varphi (b)$ (in $D$) implies that $a \t b$ (in $H$);

\item {\it cofinal} if for    every $\alpha \in D$ there is an $a \in H$ such that $\alpha \t \varphi (a)$ (in $D$);

\item a {\it divisor theory} if $D$ is free abelian, $\varphi$ is a divisor homomorphism, and for every $\alpha \in D$ there are $a_1, \ldots, a_m \in H$ such that $\alpha = \gcd \big( \varphi (a_1), \ldots, \varphi (a_m) \big)$.
\end{itemize}
In particular, every divisor theory is a cofinal divisor homomorphism. Let $\varphi \colon H \to D$ be a cofinal divisor homomorphism. The group
\[
\mathcal C (\varphi) = \mathsf q (D) / \mathsf q \big( \varphi (H) \big)
\]
is called the {\it class group} of $\varphi$. For $a \in \mathsf q (D)$ we denote by $[a] = a \mathsf q \big( \varphi (H) \big) \in \mathcal C (\varphi)$ the class containing $a$.
We use additive notation for the class group and observe that $[1]$ is the zero element of the abelian group $\mathcal C ( \varphi)$. Divisor theories of a given monoid are unique up to isomorphism. If $H$ has a divisor theory, then there is a free abelian monoid $F = \mathcal F (P)$ such that the inclusion $\varphi \colon H_{\red} \hookrightarrow F$ is a divisor theory, and the class group
\[
\mathcal C (\varphi) = \mathcal C (H) = \mathsf q (F)/\mathsf q (H_{\red})
\]
is called the {\it (divisor) class group} of $H$ and $G_0 = \{ [p] \mid p \in P \} \subset \mathcal C (H)$ is the set of classes containing prime divisors. We continue with the most classical example of a cofinal divisor homomorphism and a divisor theory.

\begin{proposition} \label{3.1}
Let $R$ be a commutative  domain, $\mathcal I^* (R)$ the monoid of invertible ideals where the operation is the usual multiplication of ideals, and let $\varphi \colon R^{\bullet} \to \mathcal I^* (R)$ be the homomorphism mapping each element onto its  principal ideal.
\begin{enumerate}
\item The map $\varphi$ is a cofinal divisor homomorphism and $\mathcal C ( \varphi)$ is the Picard group $\Pic (R)$ of $R$.

\item If $R$ is a commutative Dedekind domain, then $\varphi$ is a divisor theory and $\mathcal C (\varphi)$ is the usual ideal class group of $R$.

\item If $R$ is the ring of integers of an algebraic number field, then $\mathcal C ( \varphi)$ is finite and every class contains infinitely many prime ideals.
\end{enumerate}
\end{proposition}

\begin{proof}
1. A short calculation shows that for two invertible ideals $I, J \triangleleft R$ we have $J \t I$ in $\mathcal I^* (R)$ if and only if $I \subset J$. To show that $\varphi$ is a divisor homomorphism, let $a, b \in R^{\bullet}$ be given and suppose that $bR \t aR$ in $\mathcal I^* (R)$. Then there is a $J \in \mathcal I^* (R)$ such that $(aR)J = bR$ whence $a^{-1}bR = J \subset R$ and $a \t b$ in $R^{\bullet}$. To show that $\varphi$ is cofinal, let $I \in \mathcal I^* (R)$ be given. If $a \in I$, then $aR \subset I$ and hence $I \t aR$ in $\mathcal I^* (R)$. The definition of  $\mathcal C (\varphi)$ coincides with the definition of $\Pic (R)$.

2. Suppose that $R$ is a commutative Dedekind domain. Then every nonzero ideal is invertible and a product of prime ideals. Thus $\mathcal I^* (R)$ is free abelian. Let $I \in \mathcal I^* (R)$. Then $I$ is generated by two elements $a, b \in R$, whence $
I = \langle a, b \rangle = aR + bR = \gcd (aR, bR) = \gcd ( \varphi (a), \varphi (b))$.
Therefore $\varphi$ is a divisor theory.

3. This can be found in many textbooks on algebraic number theory (see, for example,  \cite[Theorem 2.10.14]{Ge-HK06a} for a summary).
\end{proof}

The previous proposition shows that in case of commutative Dedekind domains the embedding in a monoid of ideals establishes a divisor theory. This holds true in much greater generality and in order to outline this we mention briefly some key notions on divisorial ideals (see \cite{HK98} for a thorough treatment of divisorial ideals).

Let $H$ be a commutative monoid and let $A, B \subset \mathsf q (H)$ be subsets.
We denote by
$(A \DP B) = \{ x \in \mathsf q (H) \mid x B \subset A \}$, by $A^{-1} = (H \DP A)$, and by $A_v = (A^{-1})^{-1}$.  By an ideal of $H$ we always mean an $s$-ideal (thus  $AH = A$ holds), and an $s$-ideal $A$ is a {\it divisorial ideal} (or a {\it $v$-ideal}) if $A_v = A$. We denote by $\mathcal F_v (H)$ the set of all fractional  divisorial ideals and by $\mathcal I_v (H)$ the set of all  divisorial ideals of $H$. Furthermore, $\mathcal I_v^* (H)$ is the monoid of $v$-invertible  divisorial ideals (with $v$-multiplication) and its quotient group $\mathcal F_v (H)^{\times} = \mathsf q \big( \mathcal I_v^* (H) \big)$ is the group of fractional invertible divisorial ideals.
By  $\mathfrak X (H)$, we denote the set of all minimal nonempty prime $s$-ideals of $H$ and
\[
\widehat H = \{ x \in \mathsf q (H) \mid \ \text{there is a} \ c \in H \ \text{such that} \ cx^n \in H \ \text{for all} \ n \in \N \}  \subset \mathsf q (H)
\]
is called the {\it complete integral closure} of $H$. We say that $H$ is {\it completely integrally closed} if $H = \widehat H$. Straightforward arguments show that every factorial monoid is completely integrally closed and that a noetherian commutative domain is completely integrally closed if and only if it is integrally closed.

\begin{theorem}[{\bf Commutative Krull monoids}] \label{3.2}~

Let $H$ be a commutative monoid. Then the following statements are equivalent.
\begin{enumerate}
\item[(a)] $H$ is completely integrally closed and satisfies the {\rm ACC} on divisorial ideals.

\item[(b)] The map $\varphi \colon H \to \mathcal I_v^* (H)$,  $a \mapsto aH$ for all $a\in H$,  is a divisor theory.

\item[(c)] $H$ has a divisor theory.

\item[(d)] There is a free abelian monoid $F$ such that the inclusion $H_{\red} \hookrightarrow F$ is a divisor homomorphism.
\end{enumerate}
If one of the equivalent statements holds, then $H$ is called a {\it Krull monoid}, and $\mathcal I_v^* (H)$ is free abelian with basis $\mathfrak X (H)$.
\end{theorem}

For a proof of Theorem \ref{3.2} we refer to \cite[Section 2.5]{Ge-HK06a}. Note, since  $H$ is factorial if and only if $H_{\red}$ is free abelian, it follows that $H$ is factorial if and only if it is Krull with trivial class group.
In the remainder of this section we present a list of examples of commutative Krull monoids stemming from quite diverse mathematical areas.

\bigskip
\noindent
{\bf Commutative domains.} Let $R$ be a commutative domain and $H=R^{\bullet}$. Then the maps
\begin{equation} \label{monoid-isomorphisms}
      \iota^\bullet \colon
      \begin{cases}
      \mathcal F_v(R) &\to \  \mathcal F_v(H)\\
      \quad \mathfrak a &\mapsto \quad \mathfrak a \setminus \{0\}
      \end{cases}
      \qquad \text{and} \qquad \iota^\circ \colon
      \begin{cases}
      \mathcal F_v(H) &\to \  \mathcal F_v(R)\\
      \quad \mathfrak a &\mapsto \quad  \mathfrak a \cup \{0\}
      \end{cases}
\end{equation}
are inclusion preserving  isomorphisms which are inverse to each other. In particular, if $\mathfrak a$ is a divisorial semigroup theoretical ideal of $H$, then $\mathfrak a \cup \{0\}$ is a divisorial ring theoretical ideal of $R$. Thus $R$ satisfies the ACC on (ring theoretical) divisorial ideal of $R$ if and only if $H$ satisfies the ACC on (semigroup theoretical) divisorial ideals of $H$. Since, by definition, $R$ is completely integrally closed if and only $H$ is completely integrally closed, we obtain that $R$ is a commutative Krull domain if and only if $H$ is a commutative Krull monoid.

Property (a) in Theorem \ref{3.2} easily implies that noetherian integrally closed commutative domains are Krull. Furthermore, a commutative Krull domain is Dedekind if and only if it is at most one-dimensional. If $R$ is Dedekind, then every ideal is divisorial and $\mathcal I_v^* (R) = \mathcal I^* (R)$ (confer Theorem \ref{3.2}.(b) and Proposition \ref{3.1}.2).

\medskip
\noindent
{\bf Submonoids of commutative domains.} Let $R$ be a commutative Krull domain, $\{0\} \ne \mathfrak f \triangleleft R$ an ideal, and $\Gamma \subset (R/\mathfrak f)^{\times}$ a subgroup. Then the monoid
\[
H_{\Gamma} = \{a \in R^{\bullet} \mid a + \mathfrak f \in \Gamma\}
\]
is a Krull monoid, called the {\it (regular) congruence monoid} defined in $R$ modulo $\mathfrak f$ by $\Gamma$. We refer the reader to \cite[Section 2.11]{Ge-HK06a} for more on congruence monoids.

\medskip
\noindent
{\bf Monadic submonoids of rings of integer-valued polynomials.} Let us consider the classical ring of integer-valued polynomials over the integers. This is the ring
\[
\Int (\Z) = \{ f \in \Q [X] \mid f (\Z) \subset \Z \} \subset \Q[X] \,.
\]
We refer the reader to the Monthly article by Cahen and Chabert \cite{Ca-Ch16a} for a friendly introduction to integer-valued polynomials and to their monograph \cite{Ca-Ch97} for a deeper study. It is well-known that $\Int (\Z)$ is an integrally closed two-dimensional Pr\"ufer domain. It is a \BF-domain but it is not Krull. However, every divisor-closed submonoid of $\Int (\Z)$, which is generated by one element, is a Krull monoid \cite[Theorem 5.2]{Re14a}. We refer to recent work of Frisch and Reinhart  \cite{Re16a, Fr16a}.

\medskip
\noindent
{\bf Monoids of regular elements in commutative rings with zero-divisors.} By a commutative Krull ring we mean a completely integrally closed commutative ring which satisfies the ACC on regular divisorial ideals. The isomorphisms (as given in Equation (\ref{monoid-isomorphisms})) between  monoids of divisorial ideals carry over from the setting of commutative domains  to the setting of commutative rings with zero divisors. Thus, if a commutative ring $R$ is  Krull, then the monoid of cancelative (regular) elements is a Krull monoid, and the converse holds for $v$-Marot rings \cite[Theorem 3.5]{Ge-Ra-Re15c}.

\medskip
\noindent
{\bf Monoids of Modules.} Let $R$ be a ring and let $\mathcal C$ be a class of right $R$-modules which is closed under finite direct sums, direct summands, and isomorphisms. For a module $M$ in $\mathcal C$, let $[M]$ denote the isomorphism class of $M$. Let $\mathcal V (C)$ denote the set of isomorphism classes of modules in $\mathcal C$ (we assume here that $\mathcal V(C)$ is indeed a set.) Then $\mathcal V (\mathcal C)$ is a commutative semigroup with operation defined by $[M] + [N] = [M \oplus N]$ and all information about direct-sum decomposition of modules in $\mathcal C$ can be studied in terms of factorization of elements in the semigroup $\mathcal V (\mathcal C)$. In particular, the direct-sum decompositions in $\mathcal C$ are (essentially) unique (in other words, the Krull-Remak-Schmidt-Azumaya Theorem  holds) if and only if $\mathcal V (C)$ is a free abelian monoid. This semigroup-theoretical point of view was justified by Facchini \cite{Fa02} who showed that $\mathcal V (\mathcal C)$ is a reduced Krull monoid provided that the endomorphism ring $\End_R (M)$ is semilocal for all modules $M$ in $\mathcal C$. This result allows one to describe the direct-sum decomposition of modules in terms of factorization of elements in Krull monoids. We refer the reader to the Monthly article by Baeth and Wiegand \cite{Ba-Wi13a}.

\medskip
\noindent
{\bf Finitely generated monoids and  affine monoids.}
The {\it root closure} $\widetilde H$ of a commutative monoid $H$ is defined as
\[
\widetilde H = \{x \in \mathsf q (H) \mid x^n \in H \ \text{for some} \ n \in \N \} \subset \mathsf q (H) \,,
\]
and $H$ is said to be {\it root closed} (also the terms {\it normal}, {\it full}, and {\it integrally closed} are used) if $H = \widetilde H$. If $H$ is finitely generated, then $\widehat H = \widetilde H$. Since  finitely generated monoids satisfy the ACC on ideals, they are Krull if and only if they are root closed (see Theorem \ref{3.2}.(a)).

A monoid is called {\it affine} if it is a finitely generated submonoid of a finitely generated free abelian group. It is easy to check that  the concepts of {\it normal affine monoids} and of {\it reduced finitely generated commutative Krull monoids}  coincide (a variety of further characterizations are given in \cite[Theorem 2.7.14]{Ge-HK06a}).
(Normal) affine monoids play an important role in combinatorial commutative algebra.

\medskip
\noindent
{\bf Monoids of Zero-Sum Sequences.} Let $G$ be an additively written abelian group and $G_0 \subset G$ a subset. By a {\it sequence} over $G_0$, we mean a finite sequence of terms from $G_0$ where repetition is allowed and the order is disregarded. Clearly, the set of sequences forms a semigroup, with concatenation as its operation and with the empty sequence as its identity element. We consider sequences as elements of the free abelian monoid with basis $G_0$. This algebraic point of view has turned out to be quite convenient from a notational point of view. But there is much more which we  start to outline here and later in Proposition \ref{4.3}. Let
\[
S = g_1 \cdot \ldots \cdot g_{\ell} = \prod_{g \in G_0} g^{\mathsf v_g (S)} \in \mathcal F (G_0) \,,
\]
where $\ell \in \N_0$ and $g_1, \ldots , g_{\ell} \in G$. Then $|S|=\ell$ is the {\it length} of $S$, $\supp (S)= \{g_1, \ldots, g_{\ell} \}$ is the {\it support} of $S$, $-S= (-g_1) \cdot \ldots \cdot (-g_{\ell})$, and $\sigma (S)=g_1+ \ldots + g_{\ell}$ is the {\it sum} of $S$. We say that $S$ is a {\it zero-sum sequence} if $\sigma (S)=0$, and clearly the set
\[
\mathcal B (G_0) = \{ S \in \mathcal F (G_0) \mid \sigma (S)=0 \} \subset \mathcal F (G_0)
\]
of all zero-sum sequences is a submonoid, called the {\it monoid of zero-sum sequences} (also called Block Monoid) over $G_0$. Obviously, the inclusion $\mathcal B (G_0) \hookrightarrow \mathcal F (G_0)$ is a divisor homomorphism and hence $\mathcal B (G_0)$ is a reduced commutative Krull monoid by Theorem \ref{3.2}.(d). Monoids of zero-sum sequences form a powerful link between the theory of (general) Krull monoids and additive combinatorics \cite{Ge-Ru09, Gr13a}. Thus all methods from the later area are available for the study of sets of lengths in Krull monoids, and we will make heavily use of this in Section \ref{6}.

\begin{proposition} \label{3.3}
Let $G$ be an additive abelian group and $G_0 \subset G$ a subset.
\begin{enumerate}
\item If $G_0$ is finite, then $\mathcal B (G_0)$ is finitely generated.

\item The following statements are equivalent.
      \begin{enumerate}
       \item $|G| \le 2$.
      \item $\mathcal B (G)$ is factorial.
      \item $\mathcal B (G)$ is half-factorial.
      \end{enumerate}

\item If $|G| \ge 3$, then the inclusion $\mathcal B (G) \hookrightarrow \mathcal F (G)$ is a divisor theory with class group isomorphic to $G$ and every class contains precisely one prime divisor.

\item Let $G'$ be an abelian group. Then the monoids $\mathcal B (G)$ and $\mathcal B (G')$ are isomorphic if and only if the groups $G$ and $G'$ are isomorphic.
\end{enumerate}
\end{proposition}

\begin{proof}
1. The map $f \colon \N_0^{G_0} \to \mathcal F (G_0)$, defined by $\boldsymbol m = (m_g)_{g \in G_0} \mapsto \prod_{g \in G_0} g^{m_g}$, is a monoid isomorphism. The embedding $\Gamma := f^{-1} ( \mathcal B (G_0) ) \hookrightarrow \N_0^{G_0}$ is a divisor homomorphism (i.e, $\boldsymbol m, \boldsymbol n \in \Gamma$ and $\boldsymbol m \le \boldsymbol n$ implies that $\boldsymbol n - \boldsymbol m \in \Gamma$). By Dickson's Lemma \cite[Theorem 1.5.3]{Ge-HK06a}, $\Gamma$ is generated by the finite set of minimal points $\Min (\Gamma)$, and hence $\mathcal B (G_0)$ is generated by $f ( \Min (\Gamma))$.

2. If $G=\{0\}$, then $\mathcal B (G) = \mathcal F (G) \cong (\N_0,+)$ is free abelian. If $G = \{0, g\}$, then $\mathcal B (G)$ is free abelian with basis $\mathcal A (G) = \{0, g^2\}$. Thus (a) $\Rightarrow$ (b), and obviously (b) $\Rightarrow$ (c). In order to verify that (c) $\Rightarrow$ (a), it suffices to show that $|G| \ge 3$ implies that $\mathcal B (G)$ is not half-factorial. Suppose that $|G| \ge 3$. If there is some element $g \in G$ with $\ord (g)=n \ge 3$, then $U=g^n$, $-U$, and $V=(-g)g$ are atoms of $\mathcal B (G)$ and $(-U)U = V^n$ shows that $\mathcal B (G)$ is not half-factorial. If there are two distinct elements $e_1, e_2 \in G$ of order two, then $U=e_1e_2(e_1+e_2)$, $V_0 = (e_1+e_2)^2$, $V_1=e_1^2$, and $V_2=e_2^2$ are atoms of $\mathcal B (G)$ and $U^2=V_0V_1V_2$ shows that $\mathcal B (G)$ is not half-factorial.

3. Let $|G| \ge 3$. Clearly, the inclusion is a cofinal divisor homomorphism. To show that it is a divisor theory, let $g \in G \setminus \{0\}$ be given. If $\ord (g)=n\ge 3$, then $g = \gcd \big( g^n, g(-g) \big)$. If $\ord (g)=2$, then there is an element $h \in G \setminus \{0,g  \}$, and we obtain that $g = \gcd \big( g^{2},  g h (-g-h) \big)$. It is easy to check that the map
\[
\Phi \colon \mathcal C \big( \mathcal B (G) \big) = \mathsf q \big( \mathcal F (G) \big)/\mathsf q \big( \mathcal B (G) \big) = \{ [S]= S \mathsf q \big( \mathcal B (G) \big) \mid S \in \mathcal F (G) \} \ \to \ G \,,
\]
defined by $\Phi ( [S] ) = \sigma (S)$ is a group isomorphism. Since for every $S \in \mathcal F (G)$, $[S] \cap G = \{\sigma (S) \}$, every class of $\mathcal C \big( \mathcal B (G) \big)$ contains precisely one prime divisor.

4. This follows from 2.,3., and the fact that a reduced commutative Krull monoid is uniquely determined by its class group and the distribution of prime divisors in its classes (\cite[Theorem 2.5.4]{Ge-HK06a}).
\end{proof}

\section{Transfer Homomorphisms and Transfer Krull Monoids} \label{4}

A central method to study the arithmetic of a given class of monoids $H$ is to construct simpler auxiliary monoids $B$ (and such constructions are often based on the ideal theory of $H$) and homomorphisms $\theta \colon H \to B$ (called transfer homomorphisms) which allow us to pull back arithmetical results from $B$ to $H$.
The concept of  transfer homomorphisms was introduced by Halter-Koch in the commutative setting \cite{HK97a}) and  recently  generalized to the noncommutative setting  (\cite[Definition 2.1]{Ba-Sm15}).

\begin{definition} \label{4.1}
Let $H$ and $B$ be atomic monoids. A monoid homomorphism $\theta \colon H \to B$ is called a {\it weak transfer homomorphism} if it has the following two properties.
\begin{itemize}
\item[{\bf (T1)}] $B = B^{\times} \theta (H) B^{\times}$ and $\theta^{-1} (B^{\times})=H^{\times}$.
\item[{\bf (WT2)}] If $a \in H$, $n \in \N$, $v_1, \ldots, v_n \in \mathcal A (B)$ and $\theta (a) = v_1 \cdot \ldots \cdot v_n$, then there exist $u_1, \ldots, u_n \in \mathcal A (H)$ and a permutation $\tau \in \mathfrak S_n$ such that $a = u_1 \cdot \ldots \cdot u_n$ and $\theta (u_i) \in B^{\times} v_{\tau (i)} B^{\times}$ for each $i \in [1,n]$.
\end{itemize}
\end{definition}

Property {\bf (T1)} says that $\theta$ is surjective up to units and that only units are mapped onto units. Property {\bf (WT2)} says that factorizations can be lifted up to units and up to  order. We do not discuss equivalent formulations or variants of the definition and we do  not give the definition of a {\it transfer homomorphism}, but note that the two concepts  coincide if $H$ and $T$ are both commutative.

\begin{lemma} \label{4.2}
Let $H$ and $B$ be atomic monoids, and let $\theta \colon H \to B$ be a weak transfer homomorphism.
\begin{enumerate}
\item For every $a \in H$, we have $\mathsf L_H (a) = \mathsf L_B \big( \theta (a) \big)$. In particular, an element $a \in H$ is an atom in $H$ if and only if $\theta (a)$ is an atom in $B$.

\item $\mathcal L (H) = \mathcal L (B)$. In particular, $\Delta (H) = \Delta (B)$, $\mathcal U_k (H) = \mathcal U_k (B)$, and $\rho_k (H) = \rho_k (B)$ for every $k \in \N$.
\end{enumerate}
\end{lemma}

\begin{proof}
Since 2. follows directly from 1., we prove 1. Let $a \in H$. If $n \in \mathsf L_B \big( \theta (a) \big)$, then $\theta (a) = v_1 \cdot \ldots \cdot v_n$ with $v_1, \ldots, v_n \in \mathcal A (B)$, and thus {\bf (WT2)} implies that $n \in \mathsf L_H (a)$. Conversely, let $n \in \mathsf L_H (a)$. Then there are $u_1, \ldots, u_n \in \mathcal A (H)$ such that $a = u_1 \cdot \ldots \cdot u_n$. Thus $\theta (a) = \theta (u_1) \cdot \ldots \cdot \theta (u_n)$, and we have to verify that $\theta (u_1), \ldots,  \theta (u_n) \in \mathcal A (B)$. Let $i \in [1,n]$. Property {\bf (T1)} implies that $\theta (u_i)$ is not a unit. Since $B$ is atomic, there are $m \in \N$ and $w_1, \ldots, w_m \in \mathcal A (B)$ such that $\theta (u_i) = w_1 \cdot \ldots \cdot w_m$. Since this factorization can be lifted and $u_i$ is an atom, it follows that $m=1$ and that $\theta (u_i)=w_1 \in \mathcal A (B)$.
Since an element of an atomic monoid is an atom if and only if its set of lengths equals $\{1\}$, the  statement follows.
\end{proof}

Next we discuss the most classic example of a transfer homomorphism and its application. This is the homomorphism from a commutative Krull monoid to an associated monoid of zero-sum sequences. If $H$ is a commutative  monoid, then $H$ is Krull if and only if  $H_{\red}$ is  Krull, and if this holds, then  the canonical epimorphism $\pi \colon H \to H_{\red}$ is a transfer homomorphism. Thus, in the following proposition we may restrict to reduced Krull monoids for technical simplicity, but without loss of generality.

\begin{proposition} \label{4.3}
Let $H$ be a reduced commutative Krull monoid, $F = \mathcal F (P)$ a free abelian monoid such that the embedding $H \hookrightarrow F$ is a cofinal divisor homomorphism with class group  $G$, and let $G_0 = \{[p] \mid p \in P \} \subset G = \mathsf q (F)/\mathsf q (H)$ denote the set of classes containing prime divisors. Then there is a transfer homomorphism $\boldsymbol \beta \colon H \to \mathcal B (G_0)$. In particular, we have $\mathcal L (H) = \mathcal L \big( \mathcal B (G_0) \big)$.
\end{proposition}

\begin{proof}
Let $\widetilde{\boldsymbol \beta} \colon F \to \mathcal F (G_0)$ be the unique epimorphism defined by $\widetilde{\boldsymbol \beta} (p) = [p]$ for all $p \in P$. We start with the following assertion.

\smallskip
\begin{enumerate}
\item[{\bf A1.}\,] For every $a \in F$, we have $\widetilde{\boldsymbol \beta} (a)  \in \mathcal B (G_0)$ if and only if $a \in H$. Thus $\widetilde{\boldsymbol \beta} (H) = \mathcal B (G_0)$ and $\widetilde{\boldsymbol \beta}^{-1}  \big( \mathcal B (G_0) \big) = H$.
\end{enumerate}

\noindent
{\it Proof of} \,{\bf A1}.\,
Let $a = p_1 \cdot \ldots \cdot p_{\ell} \in F$ where $\ell \in \N_0$ and $p_1, \ldots, p_{\ell} \in P$. Then
\[
\widetilde{\boldsymbol \beta} (a) = [p_1]\cdot \ldots \cdot [p_{\ell}] \in \mathcal F (G_0) \quad \text{and} \quad \sigma \big( [p_1]\cdot \ldots \cdot [p_{\ell}] \big) = [p_1] + \ldots + [p_{\ell}] = [a] \,.
\]
Since $H \hookrightarrow F$ is a divisor homomorphism, we have $[a]=0 \in G$ if and only if $a \in H$. Therefore all assertions follow and we have proved {\bf A1}.

Therefore we can define the homomorphism $\boldsymbol \beta = \widetilde{\boldsymbol \beta}|H \colon H \to \mathcal B (G_0)$, and we assert that it is a transfer homomorphism. Clearly, $H$ and $\mathcal B (G_0)$ are reduced and  $\boldsymbol \beta$ is surjective. Thus {\bf (T1)} reads as
\[
\mathcal B (G_0) = \boldsymbol \beta (H) \quad \text{and} \quad {\boldsymbol \beta}^{-1} (\{1\}) = \{1\} \,,
\]
which holds true by {\bf A1}. We continue with the following assertion.
\smallskip
\begin{enumerate}
\item[{\bf A2.}\,] If $a \in H$, $B,C \in \mathcal B (G_0)$ and $\boldsymbol \beta (a)=BC$, then there exist $b,c \in H$ such that $a=bc$, $\boldsymbol \beta (b) = B$, and $\boldsymbol \beta (c) = C$.
\end{enumerate}

\noindent
{\it Proof of} \,{\bf A2}.\, Let  $a = p_1
\cdot \ldots \cdot p_{\ell} \in H$, where  $\ell \in \N_0$  and  $p_1,
\ldots, p_{\ell} \in P$, and suppose that $\boldsymbol \beta (a) = BC$,
say $B = [p_1] \cdot \ldots \cdot [p_k]$ and $C = [p_{k+1}] \cdot
\ldots \cdot [p_{\ell}]$ for some $k \in [0, \ell]$. By {\bf A1}, we infer that $b = p_1 \cdot \ldots \cdot p_k \in H$,
$c = p_{k+1} \cdot \ldots \cdot p_{\ell} \in H$ and clearly we have $a =
bc$. This completes the proof of {\bf A2}.

Clearly, {\bf (WT2)} follows from {\bf A2} by a straightforward induction, and hence $\boldsymbol \beta$ is a transfer homomorphism. Then Lemma \ref{4.2} implies that $\mathcal L (H) = \mathcal L \big( \mathcal B (G_0) \big)$.
\end{proof}

\begin{definition} \label{4.4}
A monoid $H$ is said to be a  {\it transfer Krull monoid} (over $G_0$) if there exists a weak transfer homomorphism $\theta \colon H \to \mathcal B (G_0)$ for a subset $G_0$ of an abelian group $G$. If $G_0$ is finite, then we say that $H$ is a {\it transfer Krull monoid of finite type}.
\end{definition}

By Proposition \ref{4.3}, every commutative Krull monoid is a transfer Krull monoid. If a monoid $H^*$ has a weak transfer homomorphism to a commutative Krull monoid, say $\theta \colon H^* \to H$, then the composition $\boldsymbol \beta \circ \theta \colon H^* \to \mathcal B (G_0)$ is a weak transfer homomorphism (with the notation of Proposition \ref{4.3}) and hence $H^*$ is a transfer Krull monoid. Thus a monoid is a transfer Krull monoid if and only if it allows a weak transfer homomorphism to a commutative Krull monoid.

Since monoids of zero-sum sequences are BF-monoids (this can be checked directly or by using Lemma \ref{2.2}), Lemma \ref{4.2} shows that transfer Krull monoids are BF-monoids. However, the examples given below reveal that transfer Krull monoids need neither be commutative nor completely integrally closed nor Mori (i.e., they do not necessarily satisfy the ACC on divisorial ideals).
Before we provide a list of transfer Krull monoids, we briefly discuss general, not necessarily commutative Krull monoids (for details we refer to \cite{Ge13a}). This concept was introduced by Wauters in 1984 in complete analogy to the ideal theoretic definition of commutative Krull monoids (compare with Theorem \ref{3.2}.(a)).

Suppose that $H$ is a monoid such that $aH \cap bH \ne \emptyset$ and $Ha \cap Hb \ne \emptyset$ for all $a,b \in H$. Then $H$ is called a {\it Krull monoid} (or a {\it Krull order}) if it is completely integrally closed  and satisfies the ACC on two-sided divisorial ideals.  The isomorphisms  in Equation (\ref{monoid-isomorphisms}) between monoids of divisorial ideals carry over from the setting of commutative domains to the setting of prime Goldie rings. Thus,
in analogy to the commutative setting, we have that a prime Goldie ring is a Krull ring if and only if its monoid of cancelative elements is a Krull monoid \cite[Proposition 5.1]{Ge13a}. Moreover, Krull monoids play a central role in the study of noetherian semigroup algebras. We refer  to \cite{C-F-G-O16} for  recent surveys on non-commutative Krull rings and monoids.

\begin{examples} \label{4.6}
{\bf 1.} As outlined above,  commutative Krull monoids (hence all the examples given is Section \ref{3})  are transfer Krull monoids. But more generally,  every normalizing Krull monoid is a transfer Krull monoid by \cite[Theorems 4.13 and 6.5]{Ge13a} (a monoid $H$ is said to be normalizing if $aH = Ha$ for all $a \in H$).

{\bf 2.} Let $H$ be a half-factorial monoid. Since the map $\theta \colon H \to \mathcal B ( \{0\})$, defined by $\theta (\epsilon)=1$ for all $\epsilon \in H^{\times}$ and $\theta (u)=0$ for every $u \in \mathcal A (H)$, is a transfer homomorphism, $H$ is a transfer Krull monoid (over the trivial group $\{0\}$). Only recently M. Roitman showed that commutative half-factorial domains need not be Mori \cite{Ro17a}. Thus,  transfer Krull monoids  satisfy the ACC on principal left ideals and on principal right ideals  (since they are BF-monoids; see Lemma \ref{2.2}) but they do not necessarily satisfy the ACC on divisorial ideals.

{\bf 3.} Let $\mathcal O$ be the ring of integers of an algebraic number field $K$, $A$ a central simple algebra over $K$, and $R$ a classical maximal $\mathcal O$-order of $A$. Then $R^{\bullet}$ is a Krull monoid. If  every stably free left $R$-ideal is free, then $R^{\bullet}$ is a transfer Krull monoid over a ray class group of $\mathcal O$ (note that this group is finite). If there is a stably free left $R$-ideal that is not free, then $R^{\bullet}$ is not a transfer Krull monoid. This is due to Smertnig \cite[Theorem 1.1 and 1.2]{Sm13a}, and for related results in a more general setting we refer to \cite{Ba-Sm15}.

{\bf 4.} Let $R$ be an order in an algebraic number field $K$,  $\overline R$  the integral closure of $R$ (thus $\overline R$ is the ring of integers of $K$), and let $\pi \colon \spec (\overline R) \to \spec (R)$ be the natural map defined by $\pi ( \mathfrak P)=\mathfrak P \cap R$ for all nonzero prime ideals $\mathfrak P \triangleleft \overline R$.

4.(a) If $R$ is seminormal, $\pi$ is bijective,  and there is an isomorphism $\overline \delta \colon \Pic (R) \to \Pic (\overline R)$, then $R^{\bullet}$ is a transfer Krull monoid over $\Pic (R)$ (\cite[Theorem 5.8]{Ge-Ka-Re15a}).

4.(b) Suppose that $\pi$ is not bijective. Since $\rho (R^{\bullet})=\infty$ by \cite[Corollary 3.7.2]{Ge-HK06a}, $R^{\bullet}$ is not a transfer Krull monoid of finite type by Theorem \ref{4.5}. Moreover, $R^{\bullet}$ is not a transfer Krull monoid over an infinite abelian group $G$ (compare Theorems \ref{4.5} and \ref{5.5}).

{\bf 5.} Let $D$ be a commutative Krull domain, $R \subset D$ a subring having the same quotient field such that $D=RD^{\times}$, $D^{\times} \cap R = R^{\times}$, and $(R \DP D)=\mathfrak m \in \max (R)$. Then the inclusion $R^{\bullet} \hookrightarrow D^{\bullet}$ is a transfer homomorphism and hence $R^{\bullet}$ is a transfer Krull monoid \cite[Proposition 3.7.5]{Ge-HK06a}. Note that $K+M$-domains satisfy the above assumptions. Indeed, let $R \subsetneq D$ be commutative  domains, $\mathfrak m$ a nonzero maximal ideal of $D$, and let $K \subsetneq L \subsetneq D$ be subfields such that $D=L+\mathfrak m$ and $R=K+\mathfrak m$. If $D$ is Krull, then the above assumptions are satisfied.

{\bf 6.} Let $R$ be a bounded HNP (hereditary noetherian prime) ring, and note that a commutative domain is an HNP ring if and only if it is a Dedekind domain. If every stably free left $R$-ideal is free, then $R^{\bullet}$ is a transfer Krull monoid \cite[Theorem 4.4]{Sm16b}.

{\bf 7.} In \cite{Ba-Ge-Gr-Sm15}, the authors study monoids of modules over HNP rings and thereby monoids of the following type occur. Let $H_0$ be a commutative Krull monoid but not a group,  $D$ be a commutative monoid with $D\ne \{1_D\}$, and define $H = (H_0 \setminus H_0^{\times})\times D \cup H_0^{\times} \times \{1_D\}$. Then $H$ is a transfer Krull monoid which is not completely integrally closed \cite[Proposition 6.1]{Ba-Ge-Gr-Sm15}.

{\bf 8.} In \cite{Ba-Ba-Go14}, the authors study conditions under which monoids of upper triangular matrices over commutative domains allow weak transfer homomorphisms to the underlying domain. Thus, in case of commutative Krull domains we obtain transfer Krull monoids.  Smertnig established  characterizations on the existence of transfer homomorphisms from full matrix rings over commutative noetherian rings with no nonzero nilpotent elements to commutative Krull domains \cite[Theorem 5.18]{Sm16a}.
\end{examples}

Sets of lengths in transfer Krull monoids (hence in all above examples) can be studied successfully with the strategy using transfer homomorphisms. Indeed combining Lemma \ref{4.2} and Proposition \ref{3.3} we are able to apply the structural results for finitely generated monoids (derived in Section \ref{2}) to transfer Krull monoids. This is done in Theorem \ref{4.5} whose proof follows  from Propositions \ref{2.8},  \ref{3.3}, and from Lemma \ref{4.2}.

\begin{theorem} \label{4.5}
Let $H$ be a transfer Krull monoid of finite type. Then the set of distances $\Delta (H)$ is finite, the elasticity $\rho (H)$ is finite, the unions $\mathcal U_k (H)$ of sets of lengths are finite for all $k \in \N$,  and they satisfy the  Structure Theorem for Unions of Sets of Lengths, as given in Theorem \ref{2.6}.
\end{theorem}

We end this section by posing the following problem (see  \cite{Ge-Sc-Zh17b}).

\begin{problem} \label{4.7}
Let $R$ be an order in an algebraic number field. Characterize when the monoid of nonzero elements $R^{\bullet}$ and when the monoid of invertible ideals $\mathcal I^* (R)$   are transfer Krull monoids, resp. transfer Krull monoids of finite type.
\end{problem}

\section{The Structure Theorem for Sets of Lengths} \label{5}

In transfer Krull monoids of finite type, not only do unions of sets of lengths do have a well-defined structure (as given in Theorem \ref{4.5}), but the same is true for sets of lengths. We start with a set of examples which demonstrate that the structure of sets of lengths is richer than that of their unions.

\begin{examples} \label{5.1}
Let $G$ be a finite abelian group and $G_0 \subset G$ a subset such that $\mathcal B (G_0)$ is not half-factorial. Since $\min \Delta \big( \mathcal B (G_0) \big) = \gcd \big( \mathcal B (G_0) \big)$ by Proposition \ref{2.3}, it follows that
for every $B \in \mathcal B (G_0)$ and every $y \in \mathsf L (B)$ we have
\[
\mathsf L (B) \subset y + d \Z \quad \text{where} \quad d = \min \Delta \big( \mathcal B (G_0) \big) \,.
\]
Clearly, every set of lengths in $\mathcal B (G_0)$ is an arithmetical progression with difference $d$ if and only if $\Delta \big( \mathcal B (G_0) \big) = \{d\}$. We will demonstrate that arithmetical progressions (of arbitrary lengths) actually occur as sets of lengths, but also several variants of arithmetical progressions do occur naturally.

{\bf 1.} {\it Arithmetical progressions.} Let $g \in G$ with $\ord (g)=n \ge 3$. Then $U = g^n$, $-U=(-g)^n$, and $V=(-g)g$ are atoms, $(-U)U=V^n$, and clearly $\mathsf L \big( (-U)U \big) = \{2,n\}$. For every $k \in \N$, we have
$\mathsf L \big( (-U)^kU^k \big) = 2k + \{\nu (n-2) \mid \nu \in [0, k]\}$.

{\bf 2.} {\it Sumsets of arithmetical progressions.} Let $r, k_1, \ldots, k_r \in \N$ and $n_1, \ldots, n_r \in \N_{\ge 3}$. For every $i \in [1,r]$, let $g_i \in G$ with $\ord (g_i)=n_i$ and we define $B_i = (-g_i)^{n_i}g_i^{n_i}$. If  $\langle g_1, \ldots, g_r \rangle = \langle g_1 \rangle \oplus \ldots \oplus \langle g_r \rangle$, then by 1., $\mathsf L (B_1^{k_1} \cdot \ldots \cdot B_r^{k_r})  =$
\[
\mathsf L (B_1^{k_1}) +  \ldots + \mathsf L ( B_r^{k_r}) \\
  = 2(k_1+ \ldots + k_r) + \sum_{i=1}^r \{\nu (n_i-2) \mid \nu \in [0, k_i]\}
\]
is the sum of $r$ arithmetical progressions. Clearly, the sum of $r$ long arithmetical progressions with differences $d_1, \ldots, d_r$ is an almost arithmetical progression with difference $d = \gcd (d_1, \ldots, d_r)$.

{\bf 3.} {\it Almost arithmetical progressions} (AAPs, see Definition \ref{2.5}). We sketch the argument that large sets of lengths in $\mathcal B (G_0)$ are AAPs with difference $d = \min \Delta \big( \mathcal B (G_0) \big)$ (for a formal statement and proof we refer to \cite[Theorem 4.3.6]{Ge-HK06a}).

We proceed as at the beginning of the proof of Theorem \ref{2.6}.
Clearly, there exist an element $C_0 \in \mathcal B (G_0)$ and $m
\in \mathbb N$ such that $\{m, m+d\} \subset \mathsf L (C_0)$. Since
$d = \gcd \Delta \big( \mathcal B (G_0 \big)$, $\psi = \rho \bigl( \Delta \big( \mathcal B (G_0 \big) \bigr) - 1 \in
\mathbb N$. Then  $L_0 = \{k_0 , k_0 + d, \ldots, k_0 + \psi d \}
\subset \mathsf L (C)$ \ where $C= C_0^{\psi}$ and  $k_0 = \psi m$. Now pick any large element $A \in \mathcal B (G_0)$, where by large we mean that $A$ is divisible by $C$. Thus, for some $B \in \mathcal B (G_0)$, we have
\[
A = BC \quad \text{and} \quad L_0 + \mathsf L (B) \subset \mathsf L (C)+\mathsf L (B) \subset \mathsf L (A) \,.
\]
Since $\mathsf L (B)$ can be viewed as an arithmetical progression with difference $d$ which has gaps (whose number is controlled by $\psi$), the sumset $\mathsf L_0 + \mathsf L (B)$ is an arithmetical progression with difference $d$. Thus, if $A$ is large (with respect to the parameters $d$ and $\psi$ depending on $G_0$), the set of lengths $\mathsf L (A)$ contains a long arithmetical progression with difference $d$ as the central part, whereas  the initial and end parts  may have gaps.

{\bf 4.} {\it Almost arithmetical multiprogressions} (the  definition is given below). Let $G_1 \subset G_0$ be a subset and let $B \in \mathcal B (G_1)$ be such that $\mathsf L (B)$ is an AAP with difference $d$, say
\[
\mathsf L (B)  = y + (L' \cup L^* \cup L'') \subset y + d \mathbb Z
\]
where $L^*$ is a long arithmetical progression with difference
$d$ (the long central part of $\mathsf L (B)$) such that $\min L^* = 0$. It is not difficult to show that every finite subset of $\N_{\ge 2}$ can be realized as a set of lengths (e.g., \cite[Proposition 4.8.3]{Ge-HK06a}). Thus for any set $\mathcal D \subset [0,d]$ with $\min \mathcal D = 0$ and $\max \mathcal D = d$, there is a zero-sum sequence $C$ with $\mathsf L (C) = x + \mathcal D$ for some $x \in \N$. Suppose that $C \in \mathcal B (G_2)$ for a subset $G_2 \subset G_0$ with $\langle G_1 \rangle \cap \langle G_2 \rangle = \{0\}$. Then
\[
\begin{aligned}
(x+y) + & \Big( (L'+\mathcal D) \uplus (L^*+\mathcal D) \uplus (L''+\mathcal D) \Big) \\
 \subset & (x+y) + \mathsf L (B)+\mathsf L (C) \subset (x+y) + \mathsf L (BC) \subset (x+y) + \mathcal D + d \Z \,.
\end{aligned}
\]
Note that the long central part $L^*+\mathcal D$ repeats the set $\mathcal D$ periodically, whereas the short initial and end parts $L'+\mathcal D$ and $L''+\mathcal D$ may contain gaps. Indeed, if $L^* = \{0, d, 2d, \ldots, \ell d \}$, then
\[
L^*+\mathcal D = \mathcal D \cup (d+\mathcal D) \cup \ldots \cup (\ell d + \mathcal D) \subset \mathcal D + d \Z \,.
\]

\end{examples}

Consider transfer Krull monoids of finite type. Then their sets of lengths coincide with sets of lengths of the  monoid of zero-sum sequences. Moreover,  the Structure Theorem for Sets of Lengths for these monoids  states that no other phenomena besides those which we have described in the above examples can occur. We  make this more precise  with the following definition.

\begin{definition} \label{5.2}
Let $d \in \N$, \ $\ell,\, M \in \N_0$ \ and \ $\{0,d\} \subset \mathcal D \subset [0,d]$. A subset $L \subset \Z$ is called an {\it almost arithmetical multiprogression} \ ({\rm AAMP} \ for
      short) \ with \ {\it difference} \ $d$, \ {\it period} \ $\mathcal D$,  {\it length} \ $\ell$ and \ {\it bound} \ $M$, \ if
\[
L = y + (L' \cup L^* \cup L'') \, \subset \, y + \mathcal D + d \Z
\]
where $y \in \Z$ is a shift parameter,
\begin{itemize}
\item the {\it central part} $L^*$ satisfies $\min L^* = 0$, $L^* = [0, \max L^*] \cap (\mathcal D + d \Z)$, and $\ell \in \N$ is maximal such that $\ell d \in L^*$,

\item   the {\it initial part} $L'$ satisfies    $L' \subset [-M, -1]$, \ and

\item the {\it end part} $L''$ satisfies $L'' \subset \max L^* + [1,M]$.
\end{itemize}
\end{definition}
Note that AAMPs are finite subsets of the integers, that an AAMP with period $\mathcal D = \{0,d\}$ is an AAP, and that an AAMP with period $\mathcal D = \{0,d\}$ and bound $M=0$ is a usual arithmetical progression with difference $d$.
As it was with AAPs (see Definition \ref{2.5}), every single finite set is an AAMP with a  trivial choice of parameters (let $L^*$ be a singleton and set $M = \max L$). To discuss one example of an AAMP (with natural parameters), let $n = p_1^{k_1} \cdot \ldots \cdot p_r^{k_r}$, where $r, k_1, \ldots, k_r \in \mathbb N$ and $p_1, \ldots, p_r$ are distinct primes.
We consider the set
$A = \{ a \in [0, n] \mid \gcd (a, n) > 1 \} \cup \{0\}$ and observe that $A = \cup_{i=1}^r p_i \mathbb N_0  \cap [0, n]$.
Setting $d = p_1 \cdot \ldots \cdot p_r$ and  $\mathcal D = A \cap [0, d]$, we obtain that
\[
A = \mathcal D + \{0, d, 2d, \ldots, (n/d-1)d \} \subset \mathcal D + d \Z
\]
is an AAMP  with difference $d$, period $\mathcal D$, and bound $M=0$.

Consider an atomic monoid with nonempty set of distances.  Lemma \ref{2.1} shows that sets of lengths become arbitrarily large.
The Structure Theorem for Sets of Lengths (formulated below) states that the set of distances is finite (whence there are only finitely many periods $\mathcal D$ with differences in $\Delta (H)$) and there is one global bound $M$ for all sets of lengths. Thus (with the above notation) long sets of lengths have a highly structured central part $L^*$, and $L^*$ is the only part of the set of lengths that can become arbitrarily large whereas the initial and end parts are universally bounded.

\begin{theorem}[{\bf Structure Theorem for Sets of Lengths}] \label{5.3}
Let $H$ be a transfer Krull monoid of finite type. Then the set of distances is finite and there is some $M \in \N_0$  such that every $L \in \mathcal L(H)$ is an {\rm AAMP} with some difference $d \in \Delta (H)$ and bound $M$.
\end{theorem}

The above theorem was first proved in \cite{Ge88} (in a slightly weaker version), and a detailed proof can be found in \cite[Chapter 4]{Ge-HK06a}. To provide an additional example for the validity of the Structure Theorem, take a commutative Mori domain $R$ with complete integral closure $\widehat R$, and with nontrivial conductor $\mathfrak f = (R \DP \widehat R)$. If the class group $\mathcal C (\widehat R)$ and the residue class ring $R/\mathfrak f$ are both  finite, then the Structure Theorem holds true \cite[Theorems 2.11.9 and 4.6.6]{Ge-HK06a} (this setting includes orders in algebraic number fields).
It is an open problem whether the assumption on the finiteness of $R/\mathfrak f$ is necessary for the validity of the Structure Theorem \cite{Ge-Ka10a, Ka16a}.
On the other hand, for transfer Krull monoids of finite type the  description given by the above   Structure Theorem  is best possible as the following realization theorem by Schmid \cite{Sc09a} shows.

\begin{theorem}[{\bf A Realization Theorem}] \label{5.4}
Let $M\in \N_0$ and $\Delta^{\ast}\subset \N$ be a finite nonempty
set. Then there exists a commutative Krull monoid $H$ with finite class group
such that the following holds{\rm \,:} for every {\rm AAMP} \ $L$ \
with difference $d\in \Delta^{\ast}$ and bound $M$ there is some
$y_{H,L} \in \mathbb N$ such that
$y+L \in \mathcal L (H) \quad \text{ for all } \quad y \ge y_{H,L}$.
\end{theorem}

We end this section with results which are in sharp contrast to the Structure Theorem. Indeed, they offer monoids where every finite subset of $\N_{\ge 2}$  occurs as a set of lengths. Moreover,  there is a transfer Krull monoid $H_1$ and a monoid $H_2$, which is not a transfer Krull monoid, whose systems of sets of lengths coincide (the first class is due to a theorem of Kainrath \cite{Ka99a} and the second example due to Frisch \cite{Fr13a}).

\begin{theorem} \label{5.5}
For the following classes of monoids we have
\[
\mathcal L (H) = \{ L \subset \N_{\ge 2} \mid L \ \text{is finite and nonempty} \} \cup \big\{ \{0\}, \{1\} \big\} \,.
\]
\begin{itemize}
\item $H$ is a transfer Krull monoid over an infinite abelian group $G$.

\item $H = \Int (\Z)^{\bullet}$ is the monoid of nonzero integer-valued polynomials over $\Z$.
\end{itemize}
Moreover, $\Int (\Z)^{\bullet}$ is not a transfer Krull monoid.
\end{theorem}

\section{The Characterization Problem for Systems of Sets of Lengths} \label{6}

Let $H$ be a transfer Krull monoid of finite type. As we have seen in Theorems \ref{4.5} and \ref{5.3}, the finite type property implies the finiteness of the set of distances and the structural results on unions of sets of lengths and on sets of lengths. In this final section we will always suppose that $H$ is a transfer Krull monoid over a finite abelian group, and this assumption will imply even  stronger results.

Thus let $H$ be a transfer Krull monoid over a finite abelian group $G$.  Then Lemma \ref{4.2} implies that $\mathcal L (H) = \mathcal L \big( \mathcal B (G) \big)$, and as usual we set $\mathcal L (G) := \mathcal L \big( \mathcal B (G) \big)$.
Recall all the examples discussed in Section \ref{3} and in Examples \ref{4.6}. In particular, rings of integers of algebraic number fields are the prototypical examples for transfer Krull monoids over finite abelian groups. Classical philosophy in algebraic number theory (dating back to the 19th century) states that the class group determines their arithmetic. This idea can be  justified (see \cite[Section 7.1]{Ge-HK06a}), and  concerning lengths of factorizations it holds true by Proposition \ref{4.3}. In the 1970s Narkiewicz posed the inverse question of whether or not arithmetical behaviour (in other words, behaviour describing the non-uniqueness of factorizations) characterize the class group.
The first affirmative answers (\cite[Sections 7.1 and 7.2]{Ge-HK06a}) have an artificial flavor because the given characterizations are based on rather abstract arithmetical properties which are designed to do the characterization and  play only a small role in other parts of factorization theory. Since on the other hand sets of lengths are of central interest in factorization theory, it has been natural to ask whether their structure is rich enough to force characterizations, and this question is known as the Characterization Problem.

\noindent
{\bf The Characterization Problem.} Given two finite abelian groups $G$ and $G'$ with $\mathsf D (G)\ge 4$ such that $\mathcal L(G) = \mathcal L(G')$.
Does it follow that $G \cong G'$?

\smallskip
Clearly, a necessary condition for an affirmative answer   is that $G$ and $G'$ are isomorphic if and only if the associated monoids $\mathcal B (G)$ and $\mathcal B (G')$ are isomorphic. This necessary condition is guaranteed by Proposition \ref{3.3}.4.
Answering the Characterization Problem is a long-term goal in the study of sets of lengths of transfer Krull monoids over finite abelian groups. We start with two elementary results (Propositions \ref{6.1} and \ref{6.2}). Then we will be in a position to  analyze the Characterization Problem in greater detail.
As usual we set $
\mathcal A (G) := \mathcal A \big( \mathcal B (G) \big), \ \Delta (G) := \Delta \big( \mathcal B (G) \big) , \ \mathcal U_k (G) = \mathcal U_k \big( \mathcal B (G) \big) , \ \text{and} \quad \rho_k (G) := \rho_k \big( \mathcal B (G) \big)$
for every $k \in \N$. Since $\mathcal A (G)$ is finite (see Propositions \ref{2.8} and \ref{3.2}), the {\it Davenport constant}
\[
\mathsf D (G) = \max \{ |U| \mid U \in \mathcal A (G) \}
\]
is finite. Clearly, $\mathsf D (G)$ is the smallest integer $\ell \in \N$ such that every sequence $S$ over $G$ of length $|S| \ge \ell$ has a zero-sum subsequence $T$ of length $|T|\ge 1$. The significance of $\mathsf D (G)$ for the study of sets of lengths will become clear in our next result.
If $|G| \le 2$, then $\mathsf D (G) = |G|$ and Proposition \ref{3.3}.2 implies that
$\mathcal L (G) = \big\{ \{k\} \mid k \in \N_0 \big\}$, whence $\Delta (G) = \emptyset$, and $\mathcal U_k (G) = \{k\}$ for every $k \in \N$. Thus  we  suppose that $2 < |G| < \infty$.

\begin{proposition} \label{6.1}
Let $G$ be a finite abelian group with $|G| \ge 3$.
\begin{enumerate}
\item For every $k \in \N$, $\mathcal U_k (G)$ is an interval, $\rho (G) = \mathsf D (G)/2$, $\rho_{2k} (G) = k \mathsf D (G)$, and
    \[
    1 + k \mathsf D (G) \le \rho_{2k+1} (G) \le k \mathsf D (G) + \lfloor \frac{\mathsf D (G)}{2} \rfloor \,.
    \]

\item $\Delta (G)$ is an interval with $\min \Delta (G)=1$ and $\max \Delta (G) \le \mathsf D (G)-2$.
\end{enumerate}
\end{proposition}

\begin{proof}
1. Let   $k \in \N$.
First, we show that $\mathcal U_k (G)$ is an interval. Note that it suffices to prove that $[k,
\rho_k (G)] \subset \mathcal U_k (G)$. Indeed, suppose that this is
done, and let $\ell \in [\min \mathcal U_k (G), k]$. Then $\ell \le k \le \rho_{\ell}
(G)$, hence $k \in \mathcal U_{\ell} (G)$ and consequently $\ell \in
\mathcal U_k (G)$.

Thus let $\ell \in [k, \rho_k (G)]$ be minimal such that $[\ell, \rho_k
(G)] \subset \mathcal U_k (G)$ and assume to the contrary that $\ell >
k$. Let $\Omega$ be the set of all $A \in \mathcal B (G)$ such that
$\{k, j \} \subset \mathsf L (A)$ for some $j \ge \ell$, and let $B \in
\Omega$ be such that $|B|$ is minimal. Then $B = U_1 \cdot \ldots
\cdot U_k = V_1 \cdot \ldots \cdot V_j$, where $j \ge \ell$ and $U_1,
\ldots, U_k, V_1, \ldots, V_j \in \mathcal A (G)$. Since $j > k$, we
have $B \ne 0^{|B|}$, and (after renumbering if necessary) we may
assume that $U_k = g_1g_2 U'$ and $V_{j-1}V_j = g_1g_2 V'$, where
$g_1, g_2 \in G$ and $U', V' \in \mathcal F (G)$. Then $U_k' =
(g_1+g_2) U' \in \mathcal A (G)$, and we suppose that $V_{j-1}' =
(g_1+g_2) V' = W_1 \cdot \ldots \cdot W_t$, where $t \in \N$ and
$W_1, \ldots , W_t \in \mathcal A (G)$. If $B' = U_1 \cdot \ldots
\cdot U_{k-1} U_k'$, then $|B'| < |B|$ and $B' = V_1 \cdot \ldots
\cdot V_{j-2}W_1 \cdot \ldots \cdot W_t$. By the minimal choice of
$|B|$, it follows that $j-2+t<\ell$, hence $t=1$, $j=\ell$ and $\ell-1 \in
\mathcal U_k (G)$, a contradiction.

Second, we study $\rho_k (G)$. We start with the following assertion.

\smallskip
{\bf A.} If $A = 0^mB \in \mathcal B (G)$, with $m = \mathsf v_0 (A) \in \N_0$ and $B \in
      \mathcal B (G)$, then
  \[
      2 \max \mathsf L (A) - m \le |A| \le \mathsf D (G) \min
      \mathsf L (A) - m ( \mathsf D (G) - 1)
      \quad \text{and} \quad
      \rho (A) \le \frac{\mathsf D (G)}{2} \,.
      \]

\noindent
{\it Proof of} \,{\bf A}.\, Let $A = 0^m U_1 \cdot \ldots \cdot U_{\ell}$ where $\ell, m \in \N_0$
and $U_1, \ldots, U_{\ell} \in \mathcal A (G) \setminus \{0\}$. Then $2 \le
|U_{\nu}| \le \mathsf D (G)$ for all $\nu \in [1, \ell]$ and hence
\[
m + 2 \ell \le |A| \le m + \ell \mathsf D (G) \,.
\]
Choosing $\ell = \min \mathsf L (B)$ and $\ell = \max \mathsf L (B)$ we
obtain the first inequalities, and then we obtain
\[
\rho (A) = \frac{\max \mathsf L (A)}{\min \mathsf L (A)} = \frac{m +
\max \mathsf L (B)}{m + \min \mathsf L (B)} \le \frac{\max \mathsf L
(B)}{\min \mathsf L (B)} \le \frac {\mathsf D (G)}{2} \,.
\]
The proof of {\bf A} and
Proposition \ref{2.4}.2  imply that $\rho_k (G) \le k \rho (G) \le k \mathsf D (G)/2$.
If $U = g_1 \cdot \ldots \cdot g_{\mathsf D (G)} \in \mathcal A (G)$, then $(-U)^kU^k= \prod_{i=1}^{\mathsf D (G)} \big( (-g_i)g_i \big)^k$, whence $k \mathsf D (G) \le \rho_{2k}(G)$ and thus $\rho_{2k}(G) = k \mathsf D (G)$. Furthermore, it follows that
\[
1+ \mathsf k \mathsf D (G) = \rho_1(G) + \rho_{2k}(G) \le \rho_{2k+1}(G) \le \frac{(2k+1)\mathsf D (G)}{2} \,.
\]
Finally, Proposition \ref{2.4}.2 implies that $\rho (G)=\mathsf D (G)/2$.

2. The proof that $\Delta (G)$ is an interval is similar  but   trickier than that of 1., and we refer to \cite{Ge-Yu12b}. It is easy to   verify that $1 \in \Delta (G)$, and we encourage the reader to do so.
Next we prove that $\max \Delta (G) \le \mathsf D (G)-2$. If $A' = 0^kA$ with $k \in \N_0$ and $A \in \mathcal B (G)$ with $0 \nmid A$, then $\mathsf L (A') = k + \mathsf L (A)$ and $\Delta ( \mathsf L (A')) = \Delta ( \mathsf L (A))$. Thus we have to prove that $\max \Delta ( \mathsf L (A)) \le \mathsf D (G)-2$ for all $A \in \mathcal B (G)$ with $0 \nmid A$, and we proceed by induction on $|A|$.
Suppose that
\[
A = U_1 \cdot \ldots \cdot U_i = V_1 \cdot \ldots \cdot V_k,  \ \text{ where }  \ i<k , \  U_1, \ldots, U_i, V_1, \ldots, V_k \in \mathcal A (G) \,,
\]
and $\mathsf L (A) \cap [i,k]=\{i,k\}$.
If $|A| \le 2 \mathsf D (G)$, then $k \le \mathsf D (G)$ and $k-i \le \mathsf D (G)-2$. Suppose that $|A| > 2 \mathsf D (G)$ and that
$\max \Delta ( \mathsf L (A')) \le \mathsf D (G)-2$ for all $A'$ with $|A'| < |A|$. If $|V_j| \ge i$ for all $j \in [1,k]$, then
\[
k i \le |V_1 \cdot \ldots \cdot V_k| = |U_1 \cdot \ldots \cdot U_i| \le i \mathsf D (G) \quad \text{and hence} \quad k-i \le \mathsf D (G)- 2 \,.
\]
Suppose that there is a $j \in [1,k]$ such that $|V_j|<i$, say $j=1$, $V_1 \t U_1 \cdot \ldots \cdot U_{i-1}$, and let $U_1 \cdot \ldots \cdot U_{i-1} = V_1 W_2 \cdot \ldots \cdot W_{\ell}$ with $\ell \in \N$ and $W_2, \ldots, W_{\ell} \in \mathcal A (G)$. We distinguish two cases.

\noindent
CASE 1: \ $\ell \ge i$.
Since $\mathsf L (U_1 \cdot \ldots \cdot U_i) \cap [i,k] = \{i,k\}$, it follows that $\mathsf L (U_1 \cdot \ldots \cdot U_{i-1}) \cap [i,k-2] = \emptyset$ and hence $\ell \ge k-1$. We may suppose that $\ell \ge k-1$ is minimal such that $U_1 \cdot \ldots \cdot U_{i-1}$ satisfies such an equation. Then the induction hypothesis implies that $\ell-(i-1) \le \mathsf D (G)-2$ and hence $k-i \le (\ell+1)-i \le \mathsf D (G)-2$.

\noindent
CASE 2: \ $\ell \le i-1$.
Since $V_1 \cdot \ldots \cdot V_k = U_1 \cdot \ldots \cdot U_i = V_1 W_2 \cdot \ldots \cdot W_{\ell}U_i$, it follows that $V_2 \cdot \ldots \cdot V_k = U_iW_2 \cdot \ldots \cdot W_{\ell}$.
Note that $\mathsf L (V_2 \cdot \ldots \cdot V_k) \cap [i, k-2] = \emptyset$, and suppose that $\ell \le i-1$ is maximal such that $V_2 \cdot \ldots \cdot V_k$ satisfies such an equation. Then the induction hypothesis implies that $(k-1)-\ell \le \mathsf D (G)-2$ and hence $k-i=(k-1)-(i-1) \le k-1-\ell \le \mathsf D (G)-2$.
\end{proof}

The state of the art on $\rho_{2k+1} (G)$ is discussed in \cite{Sc16a}.
For some small groups $G$ the system $\mathcal L (G)$ can be written down explicitly.

\medskip
\begin{proposition} \label{6.2}~

\begin{enumerate}
\item $\mathcal L (C_3) = \mathcal L (C_2 \oplus C_2) = \bigl\{ y
      + 2k + [0, k] \, \bigm| \, y,\, k \in \N_0 \bigr\}$.

\item $\mathcal L (C_4) = \bigl\{ y \negthinspace + \negthinspace k \negthinspace+1\negthinspace \negthinspace + \negthinspace [0,k] \mid y,
      k \in \N_0 \bigr\} \cup  \bigl\{ y \negthinspace +  \negthinspace 2k \negthinspace + \negthinspace 2 \cdot [0,k] \mid y, k \in \N_0 \bigr\}$.

\item $\mathcal L (C_2^3)  =  \bigl\{ y + (k+1) + [0,k] \, \bigm|\, y \in \N_0, \ k \in [0,2] \bigr\} \ \cup$ \newline
      $ \ \bigl\{ y + k + [0,k] \, \bigm|\, y \in \N_0, \ k \ge 3 \bigr\}
      \cup \bigl\{ y + 2k
      + 2 \cdot [0,k] \, \bigm|\, y ,\, k \in \N_0 \bigr\}$.

\item  $\mathcal L (C_3^2) = \{ [2k, \ell] \mid k \in \mathbb N_0, \ell \in [2k, 5k]\}$ \newline
 $\quad \text{\, } \ \qquad$ \quad $\cup \ \{ [2k+1, \ell] \mid k \in \N, \ell \in [2k+1, 5k+2] \} \cup \{ \{ 1\}  \}$.
\end{enumerate}
\end{proposition}

\begin{proof}
We prove the first statement. The proofs of the remaining statements are similar but more lengthy (details can be found in \cite[Proposition 4.2]{Ge-Sc16a}). Suppose that $G$ is cyclic of order three, say $G = \{0,g,-g\}$. Then $\mathcal A (G) = \{0, U=g^3, -U, V=(-g)g \}$, $\mathsf D (G)=3$, and  $(-U)U = V^3$ is the only minimal relation.
Clearly, $\mathsf L(V^3) = \{2,3\}$, $\mathsf L (V^{3k}) = \{2k, 2k+1, \ldots, 3k\} = 2k+[0,k]$, and $\mathsf L (0^y V^{3k}) = y+2k+[0,k]$ for all $y,k \in \N_0$. Since $\Delta (G) = \{1\}$, $\rho_{2k}(G)=3k$, and $\rho_{2k+1}(G)=3k+1$ for every $k \in \N$ by Proposition \ref{6.1}, there are no further sets of lengths.

Suppose that $G$ is an elementary $2$-group of rank two, say $G = \{0, e_1, e_2, e_1+e_2\}$. Then $\mathcal A (G) = \{0, U=e_1e_2(e_1+e_2), V_1=e_1^2, V_2=e_2^2, V_3=(e_1+e_2)^2\}$, and hence $U^2=V_1V_2V_3$ is the only minimal relation. Now the proof runs along the same lines as above.
\end{proof}

One big difficulty in all work on  the Characterization Problem stems from the fact that most sets of lengths over any finite abelian group are intervals. To make this precise we mention two deep results without proof.

\begin{theorem}[Sets of lengths which are intervals] \label{6.3}
Let $G$ be a finite abelian group with $|G| \ge 3$.
\begin{enumerate}
\item If $A$ is a zero-sum sequence whose support $\supp (A) \cup \{0\}$ is a subgroup of $G$, then  $\mathsf L (A)$ is an interval.

\item If $R$ is the ring of integers   of an algebraic number field $K$ with class group $G$, then
      \[
      \lim_{x \to \infty} \frac{ \# \{aR \mid \mathcal N_{K/\Q} (aR) \le x, \ \mathsf L (a) \ \text{{\rm is an interval}} \}}{ \# \{aR \mid \mathcal N_{K/\Q} (aR) \le x \} }  = 1 \,.
      \]
      \end{enumerate}
\end{theorem}

The first statement is a result in additive combinatorics which can be found in \cite[Theorem 7.6.8]{Ge-HK06a}. The limit formula is based on the first statement and on the analytic machinery of counting functions \cite[Theorem 9.4.11]{Ge-HK06a}. In the 1960s Narkiewicz initiated a systematic study of the asymptotic behavior of counting functions associated with non-unique factorizations.    We refer  to the monographs \cite[Chapters 7 and 9]{Na04}, \cite[Chapters 8 and  9]{Ge-HK06a}, and to  \cite{Ka16a} (analytic monoids, introduced in \cite{Ka16a}, are Krull monoids which have an abstract norm function satisfying axioms which allow to develop a theory of $L$-functions).

In spite of Theorem \ref{6.3} and
having Propositions \ref{6.1} and \ref{6.2} at our disposal, we start with a more detailed analysis of the Characterization Problem. We have  seen that
\[
\mathcal L (C_1) = \mathcal L (C_2) \quad \text{and} \quad \mathcal L (C_3) = \mathcal L (C_2 \oplus C_2) \,.
\]
An abelian group $G$ has Davenport constant $\mathsf D (G) \le 3$ if and only if it is either cyclic of order $|G| \le 3$ or isomorphic to $C_2 \oplus C_2$. Thus  we focus on groups whose Davenport constant is at least four.
Let $G$ be a finite abelian group with $\mathsf D (G) \ge 4$, say
\[
G \cong C_{n_1} \oplus \ldots \oplus C_{n_r} \ \text{ with} \ 1 < n_1 \t \ldots \t n_r \quad \text{and set} \quad \mathsf D^* (G) = 1 + \sum_{i=1}^r (n_i-1) \,.
\]
Clearly, the system $\mathcal L (G)$ depends only on $G$ and hence on the group invariants $(n_1, \ldots, n_r)$. Thus $\mathcal L (G)$ as a whole as well as the invariants controlling $\mathcal L (G)$ -- such as the set of distances $\Delta (G)$ and the $k$th elasticities $\rho_k (G)$ -- allow a description in $(n_1, \ldots, n_r)$. We demonstrate the complexity of such problems by considering $\rho_2 (G)$.

By Proposition \ref{6.1}, we have $\rho_2 (G) = \mathsf D (G)$. The Davenport constant $\mathsf D (G)$  is one of the most classical zero-sum invariants which has been studied since the 1960s. If $(e_1, \ldots, e_r)$ is a basis of $G$ with $\ord (e_i)=n_i$ for each $i \in [1,r]$, then
\[
A = (e_1+ \ldots + e_r)\prod_{i=1}^r e_i^{n_i-1} \in \mathcal A (G) \,,
\]
and hence $\mathsf D^* (G) = |A| \le \mathsf D (G)$.
It has been known since the 1960s that equality holds for $p$-groups and groups of rank at most two \cite[Theorem 5.8.3]{Ge-HK06a}. It is an open problem whether equality holds for groups of rank three, but for every $r \ge 4$ there are infinitely many groups $G$ having rank $r$ and for which $\mathsf D^* (G) < \mathsf D (G)$ holds.
We refer to \cite{Ge-Ru09, Sc16a} for a survey of what is known on parameters controlling $\mathcal L (G)$. We first show a simple finiteness result and then  present one result (a proof can be found in \cite[Theorem 6.6.3]{Ge-HK06a}) revealing characteristic phenomena of $\mathcal L (G)$ for cyclic groups and elementary $2$-groups.

\begin{lemma} \label{6.4}
Let $G$ be a finite abelian group with $\mathsf D (G) \ge 4$. Then there are only finitely many abelian groups  $G'$ (up to isomorphism) such that $\mathcal L (G) = \mathcal L (G')$.
\end{lemma}

\begin{proof}
If $G'$ is an abelian group such that $\mathcal L (G') = \mathcal L (G)$, then  Proposition \ref{6.1} implies $\mathsf D (G) \negthinspace=\negthinspace \rho_2 (G) \negthinspace=\negthinspace \rho_2 (G') \negthinspace=\negthinspace \mathsf D (G')$ and hence $ \mathsf D^* (G') \le \mathsf D (G)$.
Since there are only finitely many  $G'$ (up to isomorphism) such that $\mathsf D^* (G')$ is bounded above by a constant, there are only finitely many groups $G'$ for which $\mathcal L (G') = \mathcal L (G)$ can hold.
\end{proof}

\begin{proposition} \label{6.5}
Let $G$ be a finite abelian group with $\mathsf D (G) \ge 4$. Then $\{2, \mathsf D (G)\} \in \mathcal L (G)$ if and only if $G$ is either cyclic or an elementary $2$-group.
\end{proposition}

The next theorem gathers what is known on the Characterization Problem.

\begin{theorem} \label{6.6}
Let $G$ be a finite abelian group with $\mathsf D (G) \ge 4$, and let $G'$ be an abelian group with $\mathcal L (G) = \mathcal L (G')$. Then $G$ and $G'$ are isomorphic in each of the following cases.
\begin{enumerate}
\item $G \cong C_{n_1} \oplus C_{n_2}$ where $n_1, n_2 \in \N$ with $n_1 \t n_2$ and $n_1+n_2>4$.

\item $G$ is an elementary $2$-group.

\item $G \cong C_n^r$ where   $r, n \in \N$ with $n \ge 2$ and $2r < n-2$.

\item $\mathsf D (G) \le 11$.
\end{enumerate}
\end{theorem}

\begin{proof}[Proof of a special case]
We give a sketch of the proof for cyclic groups and for elementary $2$-groups. Let $G$ be either cyclic or an elementary $2$-group with $\mathsf D (G) \ge 4$, and let $G'$ be any abelian group with $\mathcal L (G) = \mathcal L (G')$. Since $\mathsf D (G)=\rho_2(G)=\rho_2(G')=\mathsf D (G')$ by Proposition \ref{6.1} and since $\{2, \mathsf D (G)\} \in \mathcal L (G)$ by Proposition \ref{6.5}, it follows that $\{2, \mathsf D (G')\} \in \mathcal L (G')$. Again Proposition \ref{6.5} implies that $G'$ is either cyclic  or  an elementary $2$-group. There are two proofs showing that the system of sets of lengths of cyclic groups and that of elementary $2$-groups (with the same Davenport constant) are distinct (\cite[Corollary 5.3.3, page 77]{Ge-Ru09} or \cite[Theorem 7.3.3]{Ge-HK06a}), and neither of them is elementary. Both proofs use the Savchev-Chen Structure Theorem for long zero-sum free sequences over cyclic groups (\cite[Theorem 5.1.8, page 61]{Ge-Ru09}, \cite[Chapter 11]{Gr13a}) or related statements. To discuss one approach, let $k \in \N$ and consider the inequality for $\rho_{2k+1} (G)$ given in Proposition \ref{6.1}. Elementary examples show  in case of elementary $2$-groups that we have equality on the right  side, whereas for cyclic groups we have equality on the left side (\cite[Theorem 5.3.1, page 75]{Ge-Ru09}. Detailed proofs can be found in \cite[Corollary 5.3.3, page 77]{Ge-Ru09} and \cite[Theorem 7.3.3]{Ge-HK06a}.

Now suppose that $G$ has rank two, say $G \cong C_{n_1} \oplus C_{n_2}$ where $n_1, n_2 \in \N$ with $1 < n_1 \t n_2$ and $n_1+n_2>4$, and let $G'$ be any  abelian group such that $\mathcal L (G) = \mathcal L (G')$. The proof that $G$ and $G'$ are isomorphic has two main ingredients. First, it is based on the characterization of all minimal zero-sum sequences over $G$ of length $\mathsf D (G)$. This has been done in a series of papers by Gao, Geroldinger, Grynkiewicz, Reiher, and Schmid (see \cite{B-G-G-P13a} for the characterization and detailed references). Second, it is based on the Structure Theorem for Sets of Lengths (Theorem \ref{5.3}), on an associated inverse result (\cite[Proposition 9.4.9]{Ge-HK06a}), and on a detailed study of the set of minimal distances $\Delta^* (G)$ (\cite{Ge-Zh16a}), which is defined as
\[
\Delta^* (G) = \{ \min \Delta (G_0) \mid G_0 \subset G \ \text{with} \ \Delta (G_0) \ne \emptyset \} \subset \Delta (G) \,.
\]
Detailed proofs of 1., 3., and 4. can be found in \cite{Ge-Sc16a, Ge-Zh16b, Zh17a}.
\end{proof}

We end this survey with the conjecture stating that the Characterization Problem has a positive answer for all finite abelian groups $G$ having Davenport constant $\mathsf D (G) \ge 4$.

\begin{conjecture} \label{6.7}
Let $G$ be a finite abelian group with $\mathsf D (G) \ge 4$. If $G'$ is an abelian group with $\mathcal L (G) = \mathcal L (G')$, then $G$ and $G'$ are isomorphic.
\end{conjecture}

\providecommand{\bysame}{\leavevmode\hbox to3em{\hrulefill}\thinspace}
\providecommand{\MR}{\relax\ifhmode\unskip\space\fi MR }
\providecommand{\MRhref}[2]{%
  \href{http://www.ams.org/mathscinet-getitem?mr=#1}{#2}
}
\providecommand{\href}[2]{#2}


\begin{thebibliography}{10}

\bibitem{An97}
D.D. Anderson (ed.), \emph{Factorization in {I}ntegral {D}omains}, Lect. Notes
  Pure Appl. Math., vol. 189,  Dekker, New York, 1997.

\bibitem{Ba-Ba-Go14}
D.~Bachman, N.~Baeth, and J.~Gossell, {Factorizations of upper triangular
  matrices}, \emph{Linear Algebra Appl.} \textbf{450} (2014), 138 -- 157.

\bibitem{Ba-Ge-Gr-Sm15}
N.R. Baeth, A.~Geroldinger, D.J. Grynkiewicz, and D.~Smertnig, {A
  semigroup-theoretical view of direct-sum decompositions and associated
  combinatorial problems}, \emph{J. Algebra Appl.} \textbf{14} (2015), 1550016 (60
  pages).

\bibitem{Ba-Sm15}
N.R. Baeth and D.~Smertnig, {Factorization theory: {F}rom commutative to
  noncommutative settings}, \emph{J. Algebra} \textbf{441} (2015), 475 –-- 551.

\bibitem{Ba-Wi13a}
N.R. Baeth and R.~Wiegand, {Factorization theory and decomposition of
  modules}, \emph{Amer. Math. Monthly} \textbf{120} (2013), 3 -- 34.

\bibitem{Ba-Ch11a}
P.~Baginski and S.T. Chapman, {Factorizations of algebraic integers, block
  monoids, and additive number theory}, \emph{Amer. Math. Monthly} \textbf{118}
  (2011), 901 -- 920.

\bibitem{B-G-G-P13a}
P.~Baginski, A.~Geroldinger, D.J. Grynkiewicz, and A.~Philipp, {Products
  of two atoms in {K}rull monoids and arithmetical characterizations of class
  groups}, \emph{Eur. J. Comb.} \textbf{34} (2013), 1244 -- 1268.

\bibitem{Ca-Ch97}
P.-J. Cahen and J.-L. Chabert, \emph{Integer-{V}alued {P}olynomials}, vol.~48,
  Amer. {M}ath. {S}oc. {S}urveys and {M}onographs, 1997.

\bibitem{Ca-Ch16a}
\bysame, {What you should know about integer-valued polynomials}, \emph{Amer.
  Math. Monthly} \textbf{123} (2016), 311 -- 337.

\bibitem{Ch14a}
S.T. Chapman, {A tale of two monoids: a friendly introduction to nonunique
  factorizations}, \emph{Math. Magazine} \textbf{87} (2014), 163 -- 173.

\bibitem{C-F-G-O16}
S.T. Chapman, M.~Fontana, A.~Geroldinger, and B.~Olberding (eds.),
  \emph{Multiplicative {I}deal {T}heory and {F}actorization {T}heory},
  Proceedings in Mathematics and Statistics, vol. 170, Springer, 2016.

\bibitem{Ch-Sm90a}
S.T. Chapman and W.W. Smith, {Factorization in {D}edekind domains with
  finite class group}, \emph{Isr. J. Math.} \textbf{71} (1990), 65 -- 95.

\bibitem{Co-Sm11a}
J.~Coykendall and W.W. Smith, {On unique factorization domains}, \emph{J.
  Algebra} \textbf{332} (2011), 62 -- 70.

\bibitem{numericalsgps}
M.~Delgado, P.A. Garc{\'i}a-S{\'a}nchez, and J.~Morais,
  \emph{``numericalsgps'': a {\sf {g}{a}{p}} package on numerical semigroups},
  \verb+(http://www.gap-system.org/Packages/numericalsgps.html)+.

\bibitem{Fa02}
A.~Facchini, {Direct sum decomposition of modules, semilocal endomorphism
  rings, and {K}rull monoids}, \emph{J. Algebra} \textbf{256} (2002), 280 -- 307.

\bibitem{Fo-Ho-Lu13a}
M.~Fontana, E.~Houston, and T.~Lucas, \emph{Factoring {I}deals in {I}ntegral
  {D}omains}, Lecture Notes of the Unione Matematica Italiana, vol.~14,
  Springer, 2013.

\bibitem{Fr13a}
S.~Frisch, {A construction of integer-valued polynomials with prescribed
  sets of lengths of factorizations}, \emph{Monatsh. Math.} \textbf{171} (2013), 341
  -- 350.

\bibitem{Fr16a}
\bysame, {Relative polynomial closure and monadically {K}rull monoids of
  integer-valued polynomials}, in \emph{Multiplicative {I}deal {T}heory and
  {F}actorization {T}heory}. Ed. S.~T. Chapman, M.~Fontana, A.~Geroldinger, and
  B.~Olberding. Springer, 2016, pp.~145 -- 157.

\bibitem{Ga-Ge09b}
W.~Gao and A.~Geroldinger, {On products of $k$ atoms}, \emph{Monatsh. Math.}
  \textbf{156} (2009), 141 -- 157.

\bibitem{GS16a}
P.A. Garc{\'i}a-S{\'a}nchez, {An overview of the computational aspects of
  nonunique factorization invariants}, in \emph{Multiplicative {I}deal {T}heory and
  {F}actorization {T}heory}. Ed. S.T. Chapman, M.~Fontana, A.~Geroldinger, and
  B.~Olberding. Springer, 2016, pp.~159 -- 181.

\bibitem{Ge88}
A.~Geroldinger, {{\"{U}}ber nicht-eindeutige {Z}erlegungen in irreduzible
  {E}lemente}, \emph{Math. Z.} \textbf{197} (1988), 505 -- 529.

\bibitem{Ge13a}
\bysame, {Non-commutative {K}rull monoids: a divisor theoretic approach
  and their arithmetic}, \emph{Osaka J. Math.} \textbf{50} (2013), 503 -- 539.

\bibitem{Ge-HK06a}
A.~Geroldinger and F.~Halter-Koch, \emph{Non-{U}nique {F}actorizations.
  {A}lgebraic, {C}ombinatorial and {A}nalytic {T}heory},  Chapman \& Hall/CRC, Boca Raton, FL,  2006.

\bibitem{Ge-Ka10a}
A.~Geroldinger and F.~Kainrath, {On the arithmetic of tame monoids with
  applications to {K}rull monoids and {M}ori domains}, \emph{J. Pure Appl. Algebra}
  \textbf{214} (2010), 2199 -- 2218.

\bibitem{Ge-Ka-Re15a}
A.~Geroldinger, F.~Kainrath, and A.~Reinhart, {Arithmetic of seminormal
  weakly {K}rull monoids and domains}, \emph{J. Algebra} \textbf{444} (2015), 201 --
  245.

\bibitem{Ge-Le90}
A.~Geroldinger and G.~Lettl, {Factorization problems in semigroups},
  \emph{Semigroup Forum} \textbf{40} (1990), 23 -- 38.

\bibitem{Ge-Ra-Re15c}
A.~Geroldinger, S.~Ramacher, and A.~Reinhart, {On $v$-{M}arot {M}ori rings
  and $\rm{C}$-rings}, \emph{J. Korean Math. Soc.} \textbf{52} (2015), 1 -- 21.

\bibitem{Ge-Ru09}
A.~Geroldinger and I.~Ruzsa, \emph{Combinatorial {N}umber {T}heory and
  {A}dditive {G}roup {T}heory}, Advanced Courses in Mathematics - CRM
  Barcelona, Birkh{\"a}user, 2009.

\bibitem{Ge-Sc16a}
A.~Geroldinger and W.~A. Schmid, {A characterization of class groups via
  sets of lengths}, {http://arxiv.org/abs/1503.04679}.

\bibitem{Ge-Sc-Zh17b}
A.~Geroldinger, W.~A. Schmid, and Q.~Zhong, {Systems of sets of lengths:
  transfer {K}rull monoids versus weakly {K}rull monoids},
  {http://arxiv.org/abs/1606.05063}.

\bibitem{Ge-Yu12b}
A.~Geroldinger and P.~Yuan, {The set of distances in {K}rull monoids},
  \emph{Bull. Lond. Math. Soc.} \textbf{44} (2012), 1203 –-- 1208.

\bibitem{Ge-Zh16b}
A.~Geroldinger and Q.~Zhong, {A characterization of class groups via sets
  of lengths {II}}, \emph{J. Th{\'e}or. Nombres Bordx.} (forthcoming).

\bibitem{Ge-Zh16a}
\bysame, {The set of minimal distances in {K}rull monoids}, \emph{Acta Arith.}
  \textbf{173} (2016), 97 -- 120.

\bibitem{Gr13a}
D.J. Grynkiewicz, \emph{Structural {A}dditive {T}heory}, Developments in
  Mathematics, Springer, 2013.

\bibitem{HK97a}
F.~Halter-Koch, {Finitely generated monoids, finitely primary monoids and
  factorization properties of integral domains}, in \emph{Factorization in {I}ntegral
  {D}omains}. Ed. D.D. Anderson.  Lect. Notes Pure Appl. Math., vol. 189,  Dekker, New York, 1997,
  pp.~73 -- 112.

\bibitem{HK98}
\bysame, \emph{Ideal {S}ystems. {A}n {I}ntroduction to {M}ultiplicative {I}deal
  {T}heory}, Marcel Dekker, 1998.

\bibitem{H-L-S-S15}
G.~Heged{\"u}s, Z.~Li, J.~Schicho, and H.P. Schr{\"o}cker, {From the
  fundamental theorem of algebra to {K}empe's universality theorem}, \emph{Internat.
  Math. Nachrichten} \textbf{229} (2015), 13--26.

\bibitem{Ja65}
B.~Jacobson, {Matrix number theory{\rm \,:} an example of non-unique
  factorization}, \emph{Amer. Math. Monthly} \textbf{72} (1965), 399 -- 402.

\bibitem{Ka16a}
J.~Kaczorowski, {Analytic monoids and factorization problems}, \emph{Semigroup
  Forum} (forthcoming).

\bibitem{Ka99a}
F.~Kainrath, {Factorization in {K}rull monoids with infinite class group},
  \emph{Colloq. Math.} \textbf{80} (1999), 23 -- 30.

\bibitem{Ma-Ok16a}
P.~Malcolmson and F.~Okoh, {Half-factorial subrings of factorial domains},
  \emph{J. Pure Appl. Algebra} \textbf{220} (2016), 877 –-- 891.

\bibitem{Na04}
W.~Narkiewicz, \emph{Elementary and {A}nalytic {T}heory of {A}lgebraic
  {N}umbers, 3rd ed.}, Springer, 2004.

\bibitem{Re16a}
A.~Reinhart, {On the divisor-class group of monadic submonoids of rings of
  integer-valued polynomials}, \emph{Commun. Korean Math. Soc.} (forthcoming).

\bibitem{Re14a}
\bysame, {On monoids and domains whose monadic submonoids are {K}rull}, in \emph{
  Commutative {A}lgebra. {R}ecent {A}dvances in {C}ommutative {R}ings,
  {I}nteger-{V}alued {P}olynomials, and {P}olynomial {F}unctions}. Ed.  M.~Fontana,
  S.~Frisch, and S.~Glaz. Springer, 2014, pp.~307 -- 330.

\bibitem{Ro17a}
M.~Roitman, {Half-factorial domains}, manuscript.

\bibitem{Sc09a}
W.A. Schmid, {A realization theorem for sets of lengths}, \emph{J. Number Theory}
  \textbf{129} (2009), 990 -- 999.

\bibitem{Sc16a}
\bysame, {Some recent results and open problems on sets of lengths of
  {K}rull monoids with finite class group}, in \emph{Multiplicative {I}deal {T}heory and
  {F}actorization {T}heory}. Ed. S.T. Chapman, M.~Fontana, A.~Geroldinger, and
  B.~Olberding. Springer, 2016, pp.~323 -- 352.

\bibitem{Sm16b}
D.~Smertnig, {{Factorizations in bounded hereditary noetherian prime
  rings}}, {http://arxiv.org/abs/1605.09274}.

\bibitem{Sm13a}
\bysame, {Sets of lengths in maximal orders in central simple algebras},
  \emph{J. Algebra} \textbf{390} (2013), 1 -- 43.

\bibitem{Sm16a}
\bysame, {Factorizations of elements in noncommutative rings: {A}
  {S}urvey}, in \emph{Multiplicative {I}deal {T}heory and {F}actorization {T}heory}. Ed.
  S.T. Chapman, M.~Fontana, A.~Geroldinger, and B.~Olberding. Springer,
  2016, pp.~353 -- 402.

\bibitem{Zh17a}
Q.~Zhong, {Sets of minimal distances and characterizations of class groups
  of {K}rull monoids}, {http://arxiv.org/abs/1606.08039}.

\end{thebibliography}
\end{document}